\documentclass[12pt]{amsart}
\usepackage{txfonts}
\usepackage{amscd,amssymb,amsmath}
\usepackage[all, cmtip]{xy}
\usepackage{graphicx,calrsfs}
\usepackage[pdftex,colorlinks,plainpages,backref,urlcolor=blue]{hyperref}
\usepackage{url}
\usepackage{array}
\usepackage[titletoc]{appendix}

\newtheorem{theorem}{Theorem}[section]
\newtheorem{lemma}[theorem]{Lemma}
\newtheorem{corollary}[theorem]{Corollary}
\newtheorem{proposition}[theorem]{Proposition}

\theoremstyle{definition}

\newtheorem{example}[theorem]{Example}
\newtheorem{remark}[theorem]{Remark}

\newcommand{\R}{\mathbb{R}}
\newcommand{\C}{\mathbb{C}}

\renewcommand{\k}{\Bbbk}
\newcommand{\K}{\mathbb{K}}

\newcommand{\bL}{\mathbf{L}}

\newcommand{\RR}{{\mathcal R}}

\newcommand{\m}{{\mathfrak{m}}}

\DeclareMathOperator{\rank}{rank}

\DeclareMathOperator{\id}{id}

\DeclareMathOperator{\Sym}{Sym}
\DeclareMathOperator{\ch}{char}
\DeclareMathOperator{\GL}{GL}

\DeclareMathOperator{\Tor}{{Tor}}

\DeclareMathOperator{\corank}{corank}

\DeclareMathOperator{\pf}{pf}
\DeclareMathOperator{\Pf}{Pf}
\DeclareMathOperator{\Det}{Det}

\DeclareMathOperator{\nl}{null}
\DeclareMathOperator{\PD}{PD}

\renewcommand{\emptyset}{\mathrm \O}

\newcommand{\surj}{\twoheadrightarrow}
\newcommand{\inj}{\hookrightarrow}

\def\dot{\mathchar"013A}  
\newcommand{\hdot}{{\raise1pt\hbox to0.35em{\Large $\dot$\!}}} 
\newcommand{\bwedge}{\mbox{\normalsize $\bigwedge$}}

\newcommand{\isom}{\xrightarrow{
   \,\smash{\raisebox{-0.65ex}{\ensuremath{\scriptstyle\simeq}}}\,}}

\DeclareMathAlphabet\mathbfcal{OMS}{cmsy}{b}{n}

\makeatletter
\def\@tocline#1#2#3#4#5#6#7{\relax
  \ifnum #1>\c@tocdepth 
  \else
    \par \addpenalty\@secpenalty\addvspace{#2}%
    \begingroup \hyphenpenalty\@M
    \@ifempty{#4}{%
      \@tempdima\csname r@tocindent\number#1\endcsname\relax
    }{%
      \@tempdima#4\relax
    }%
    \parindent\z@ \leftskip#3\relax \advance\leftskip\@tempdima\relax
    \rightskip\@pnumwidth plus4em \parfillskip-\@pnumwidth
    #5\leavevmode\hskip-\@tempdima
      \ifcase #1
       \or\or \hskip 1em \or \hskip 2em \else \hskip 3em \fi%
      #6\nobreak\relax
    \dotfill\hbox to\@pnumwidth{\@tocpagenum{#7}}\par
    \nobreak
    \endgroup
  \fi}
\makeatother

\newcommand{\downmapsto}{\rotatebox[origin=c]{-90}{$\scriptstyle\mapsto$}\mkern2mu}

\numberwithin{equation}{section}

\title[Poincar\'{e} duality and resonance varieties]%
{Poincar\'{e} duality and resonance varieties}

\author[Alexander~I.~Suciu]{Alexander~I.~Suciu$^1$}
\address{Department of Mathematics,
Northeastern University,
Boston, MA 02115, USA}
\email{\href{mailto:a.suciu@northeastern.edu}{a.suciu@northeastern.edu}}
\urladdr{\href{http://web.northeastern.edu/suciu/}%
{web.northeastern.edu/suciu/}}
\thanks{$^1$Supported in part by the Simons Foundation Collaboration Grant 
for Mathematicians \#354156}



\begin{document}

\begin{abstract}
We explore the constraints imposed by Poincar\'{e} duality on 
the resonance varieties of a graded algebra.  For a $3$-dimensional 
Poincar\'{e} duality algebra $A$, we obtain a fairly precise geometric 
description of the resonance varieties $\RR^i_k(A)$.  
\end{abstract}

\maketitle
\tableofcontents

\section{Introduction}
\label{sect:intro}

\subsection{Resonance varieties}
\label{subsec:res-intro}
The cohomology ring of a space captures deep, albeit incomplete information 
about the homotopy type of the space.  Suppose we are given a connected, 
finite CW-complex $X$ and a coefficient field $\k$ of characteristic different from $2$. 
Finding a presentation for the $\k$-algebra $A=H^{\hdot}(X,\k)$, in 
and of itself, is not the end of the story.  One still would like to extract 
further information from this graded algebra, such as the Betti numbers, 
$b_i(A)=\dim_{\k} A^i$, the bigraded Betti numbers $b_{ij}=\dim_{\k} \Tor^A_i(\k,\k)_j$, 
or the cup-length.  Such numerical invariants, though, are oftentimes too 
coarse to tell apart graded algebras which may differ in quite subtle ways.

Enter the resonance varieties, $\RR^i_k(A)$, which are the main focus 
of attention in this paper.  These varieties are homogenous algebraic 
subsets of the affine space $A^1=H^1(X,\k)$ which keep track of vanishing 
cup products in the cohomology ring of $X$.  More precisely, for each 
$a\in A^1$, consider the cochain complex  $(A,\delta_a)$ with differentials 
$\delta^i_a\colon A^i\to A^{i+1}$ given by $\delta^i_a(u)=au$. 
Then the degree $i$, depth $k$ resonance variety $\RR^i_k(A)$ 
consists of those points $a\in A^1$ for which 
$H^i(A,\delta_a)$ has dimension at least $k$. In particular, 
$\RR^1_1(A)$ is the union of all isotropic planes in $A^1$. 

In general, the resonance varieties can be quite complicated.  
On the other hand, if $A$ is the cohomology ring of a formal space, 
then the resonance varieties of $A$ are unions of rationally defined, 
linear subspaces of $A^1$, see \cite{DPS-duke, DP-ccm}. 
Our main goal here is to see what kind of restrictions another topological 
property, namely, Poincar\'{e} duality, puts on the resonance varieties.

\subsection{Poincar\'{e} duality algebras}
\label{subsec:pd-intro}
A graded, locally finite, graded commutative algebra $A$ 
is said to be a Poincar\'{e} duality algebra of dimension $m$ if 
there exists a $\k$-linear map $\varepsilon\colon A^m \to \k$ 
such that all the bilinear forms $A^i \otimes A^{m-i}\to \k$, 
$a\otimes b\mapsto \varepsilon (ab)$ are non-singular. 
For such a $\PD_m$ algebra, the Betti numbers satisfy the well-known 
equality $b_i(A)=b_{m-i}(A)$.  A similar phenomenon holds 
for the resonance varieties; more precisely, we show in 
Theorem \ref{thm:rpdf} that 
\begin{equation}
\label{eq:pdres-intro}
\RR^{i}_k(A)=\RR^{m-i}_k(A), 
\end{equation}
for all $i$ and $k$.   Most interesting to us is the case when $m=3$.  
For a $\PD_3$ algebra $A$, we have that $\RR^{1}_k(A)=\RR^2_k(A)$, and 
$\RR^i_k(A)\subseteq \{0\}$ for $i=0$ or $3$. So we are left with computing 
the degree $1$ resonance varieties.  

To that effect, we start by noting that the multiplicative  
structure of $A$ is encoded by the alternating $3$-form 
$\mu_A\colon \bwedge^3 A^1\to \k$ 
given by $\mu_A(a\wedge b \wedge c)= \varepsilon (abc)$. 
Fixing a basis $\{e_1,\dots ,e_n\}$ for $A^1$, and 
setting $\mu_{ijk}=\mu_A(e_i\wedge e_j\wedge e_k )$, 
this information can be stored dually in the trivector 
$\mu=\sum \mu_{ijk}\, e^i \wedge e^j \wedge e^k$ belonging to 
$\bwedge^3 (A^1)^*$. 
Conversely, any  $3$-form $\mu \colon \bwedge^3 V \to \k$ 
on a finite-dimensional $\k$-vector space $V$ determines in an 
obvious fashion a $\PD_3$ algebra $A$ over $\k$ for which $\mu_A=\mu$.  
As shown in Theorem \ref{thm:pd3-iso}, 
this construction yields a one-to-one correspondence, 
$A \leftrightsquigarrow \mu_A$, between isomorphism 
classes of $\PD_3$ algebras and equivalence classes of alternating $3$-forms.

The rank of a $3$-form $\mu$ is the minimum dimension of a linear subspace 
$W\subset V$ such that $\mu$ factors through $\bigwedge^3 W$.
The computation of the degree $1$ resonance varieties of a $\PD_3$ 
algebra reduces to the case when the associated $3$-form 
has maximal rank.   More precisely, let $A$ be any $\PD_3$ 
algebra, and write $A^1=B^1\oplus C^1$, where the restriction 
of $\mu_A$ to $\bigwedge^3 B^1$ has rank equal to the rank of $\mu_A$. 
Letting $B$ the $\PD_3$ algebra with associated $3$-form equal to this restriction, 
we show in Theorem \ref{thm:nonmax} that 
\begin{equation}
\label{eq:r1decomp}
\RR^1_k(A) \cong  \RR^1_{k-r+1}(B) \times C^1
\cup \RR^1_{k-r}(B) \times \{0\}
\end{equation}
for all $k\ge 0$, where $r=\corank \mu_A$. In particular, $\RR^1_k(A)=A^1$ for 
all $k< \corank \mu_A$.

In Theorem \ref{thm:dimres} we give a lower bound on the dimension of the  
degree-$1$ resonance varieties up to a certain depth.  Letting $\nu$ denote 
the nullity of $\mu_A$, we show that 
\begin{equation}
\label{eq:resbound}
\dim \RR^1_{\nu -1}(A)\ge \nu \ge 2,
\end{equation}
provided $\overline{\k}=\k$ and $b_1(A)\ge 4$; in particular, $\dim \RR^1_1(A)\ge \nu$. 
Finally, in Theorem \ref{thm:real isotropic} we use a result from \cite{DrS10} to 
show that, with a few exceptions, $\RR^1_1(A)\ne \{0\}$, provided $\k=\R$. 

\subsection{Pfaffians and resonance}
\label{subsec:pfaff-intro}
Consider now the polynomial ring $S=\k[x_1,\dots, x_n]$, and let 
$\theta$ be the $n\times n$ skew-symmetric matrix of $S$-linear 
forms with entries $\theta_{ik}=\sum_{j=1}^n \mu_{jik} x_j$.  It turns 
out that the resonance varieties of $A$ are the degeneracy loci of 
this matrix, that is, 
\begin{equation}
\label{eq:minors-intro}
\RR^1_k(A)=V(I_{n-k}(\theta)),
\end{equation}
the vanishing locus of the ideal of codimension $k$ minors of $\theta$.  
Using known facts about Pfaffian ideals of skew-symmetric matrices, 
we show in Theorem \ref{theorem:be} that 
\begin{equation} 
\label{eq:res jumps-intro}
\RR^1_{2k}(A) =\begin{cases}
\RR^1_{2k+1}(A)& \text{if $n$ is even},\\[2pt]
\RR^1_{2k-1}(A) & \text{if $n$ is odd}.
\end{cases}
\end{equation}

We also show in Theorem \ref{thm:rvanish} that the bottom resonance 
varieties vanish, provided $n\ge 3$ and $\mu_A$ has maximal rank: 
\begin{equation}
\label{eq:bottom-intro}
\RR^1_{n-2}(A)=\RR^1_{n-1}(A)=\RR^1_{n}(A)=\{0\}.
\end{equation}
In this case, we have the following chains  
of inclusions for the varieties $\RR_k=\RR^1_k(A)$:
\begin{align}
\label{eq:rinc}
&A^1=\RR_0=\RR_1 \supseteq \RR_2=\RR_3\supseteq 
 \cdots \supseteq \RR_{n-3} \supseteq \RR_{n-2} =\{0\} &&\text{if $n$ is even},
\\    \notag
&A^1=\RR_0\supseteq \RR_1= \RR_2\supseteq \RR_3 
 \supseteq \cdots \supseteq \RR_{n-3} \supseteq \RR_{n-2} = \{0\} 
 &&\text{if $n$ is odd}.
\end{align}

\subsection{The top resonance varieties}
\label{subsec:topbottom-intro}
By way of contrast, the top resonance varieties of a $\PD_3$ algebra 
$A$ have a much more interesting geometry.  
Without essential loss of generality, we may assume that $n=\dim A^1$ 
is at least $4$ (the cases when $n\le 3$ are easily dealt with). 
We then show in Theorem \ref{thm:respd3} that 
\begin{equation}
\label{eq:r1 pd3-intro}
\RR^1_1(A)=
\begin{cases}
V(\Pf(\mu_A))  & \text{if\/ $n$ is odd and $\mu_A$ is generic in the 
sense of \cite{BP}},\\
A^1 & \text{otherwise}.
\end{cases}
\end{equation}

Finally, suppose $\mu_A$ is generic in the sense of \cite{DFMR}.  
If $n$ is odd, then $\RR^1_1(A)$ is a hypersurface which is smooth if 
$n\le 7$, and singular in codimension $5$ if $n\ge 9$. On the other hand, if 
$n$ is even, then $\RR^1_2(A)$ is a subvariety of codimension $3$, 
which is smooth if $n\le 10$, and is singular in codimension $7$ if $n\ge 12$.

In Appendix \ref{sect:tables} we list the irreducible $3$-forms $\mu_A$ 
of rank at most $8$, according to the classification from \cite{Gu,Dj1},  
together with the corresponding  resonance varieties $\RR^1_k(A)$.  

This work is pursued in \cite{Su-tcone3d}, where we provide 
further applications to the study of cohomology jump loci of $3$-manifolds. 

\section{The resonance varieties of a graded algebra}
\label{sect:resvar}

\subsection{Resonance varieties}
\label{subsec:rv}
Let $A$ be a graded, graded commutative algebra over 
a field $\k$ of characteristic different from $2$.  
Throughout, we will assume that $A$ is non-negatively graded, 
that $A$ is of finite-type (i.e., each graded piece $A^i$ 
is finite-dimensional), and that $A$ is connected (i.e., 
$A^0=\k$, generated by the unit $1$).  We 
will write $b_i=b_i(A)$ for the Betti numbers of $A$, 
and we will generally assume that 
$b_1>0$, so as to avoid trivialities. 

By graded-commutativity of the product and the assumption that $\ch \k\ne 2$, 
each element $a\in A^1$ squares to zero. We thus obtain a cochain complex,
\begin{equation}
\label{eq:aomoto}
\xymatrix{(A , \delta_a)\colon  \ 
A^0\ar^(.66){\delta^0_a}[r] & A^1\ar^{\delta^1_a}[r] & A^2 \ar^{\delta^2_a}[r]  
& \cdots},
\end{equation}
with differentials $\delta^i_a(u)=a\cdot u$, for all $u\in A^i$.    
The {\em resonance varieties}\/ of $A$ (in degree $i\ge 0$ 
and depth $k\ge 0$) are defined as 
\begin{equation}
\label{eq:rvs}
\RR^i_k(A)=\{a \in A^1 \mid \dim_{\k} H^i(A,a) \ge k\}. 
\end{equation}

In other words, the resonance varieties record the locus of points $a$ 
in the affine space $A^1=\k^{b_1}$ where the `twisted' Betti numbers 
$b_i(A,a):=\dim_{\k} H^i(A,\delta_a)$ jump by at least $k$.   
We will allow at times $k\le 0$, 
in which case we will set  $\RR^i_k(A)=A^1$.  Clearly, the sets $\RR^i_k(A)$ 
are homogeneous subsets of $A^1$.  Here is a more concrete  
description of these sets, which follows at once from the definitions.

\begin{lemma}
\label{lem:res}
An element $a\in A^1$ belongs 
to $\RR^i_k(A)$ if and only if there exist 
$u_1,\dots ,u_k\in A^i$ such that  $au_1=\cdots =au_k=0$ 
in $A^{i+1}$, and the set $\{au,u_1,\dots ,u_k\}$ 
is linearly independent in $A^i$, for all $u\in A^{i-1}$.
\end{lemma}

Consequently, $\RR^i_{b_i}(A)=\{0\}$ and 
$\RR^i_k(A)= \emptyset$ for $k>b_i$; in particular, 
if $b_1=0$, then $\RR^1_{k}(A)=\emptyset$ for all $k\ge 1$. 
Moreover, for each $i \ge 0$, we have a descending filtration, 
\begin{equation}
\label{eq:resfilt}
A^1=\RR^i_0(A)\supseteq \RR^i_1(A)\supseteq \cdots 
\supseteq \RR^i_{b_i} (A)=\{0\}\supset \RR^i_{b_{i+1}} (A) =\emptyset.
\end{equation}
Therefore, 
\begin{equation}
\label{eq:bia}
b_i(A)=\max \big\{ k \mid 0\in \RR^i_k(A)\big\}\, .
\end{equation} 

\subsection{Isotropic subspaces}
\label{subsec:isotropic}
We say that a linear subspace $U\subset A^1$ is {\em isotropic}\/ 
if the restriction of the multiplication map 
$A^1\wedge A^1\to A^2$ to $U\wedge U$ is the zero map; 
that is, $ab=0$, for all $a,b\in U$.

\begin{lemma}
\label{lem:isotropic}
Let $A$ be a graded algebra as above.
\begin{enumerate}
\item \label{res1}
If $U\subseteq A^1$ is an isotropic subspace of dimension $k$, then 
$U\subseteq \RR^1_{k-1}(A)$. 

\item \label{res2}
$\RR^1_1(A)$ is the union of all isotropic planes 
in $A^1$. 
\end{enumerate}
\end{lemma}

\begin{proof}
The first claim follows straight from the definitions. 
To prove claim \eqref{res2},  let $\mathcal{Q}(A)$ 
be the union of all isotropic planes in $A^1$. 
By claim \eqref{res1}, we have that $\mathcal{Q}(A)\subseteq \RR^1_1(A)$; 
it remains to establish the reverse inclusion. 

So let $a\in \RR^1_1(A)$; there is then a vector $b\in A^1$, 
not proportional to $a$, such that $ab=0$ in $A^2$. Let $U$ be the plane 
spanned by $a$ and $b$. Then $U$ is isotropic  
(if $\alpha=\lambda_1 a+\nu_1 b$ and $\beta=\lambda_2 a+\nu_2 b$ 
are two vectors in $U$, then clearly $\alpha\beta=0$), and we are done.
\end{proof}

\begin{remark}
\label{rem:falk}
The resonance varieties $\RR^1_1(A)$ were first considered by Falk \cite{Fa97} 
in the case when $A$ is the Orlik--Solomon algebra attached to a hyperplane 
arrangement and $\k=\C$.  It was noted in that paper that Lemma \ref{lem:isotropic} 
holds in that setting, while subsequent  work of Falk \cite{Fa07} highlighted and
 made use of the fact that these rulings by isotopic planes hold over 
fields $\k$ of arbitrary characteristic, even when $\RR^1_1(A)$ is not a union 
of linear subspaces, as is the case when $\ch(\k)=0$.
\end{remark}

\begin{remark}
\label{rem:char-res}
The resonance varieties of a graded algebra $A$ 
do not depend in an essential way on the field $\k$, but rather, just on its 
characteristic. More precisely, if $\k\subset \K$ is a field extension, then 
the $\k$-points on $\RR^i_k(A\otimes_\k \K)$ coincide with $\RR^i_k(A)$.  
Nonetheless, as we shall see in Example \ref{ex:sing-iso}, this subtle difference 
between the two varieties can be quite meaningful.
\end{remark}

\subsection{Resonance varieties of products}
\label{subsec:resprod}

One of the more pleasant properties of resonance varieties 
is the way they behave with respect to tensor products 
of graded algebras.  This topic is treated in various 
levels of generality in  \cite{PS-plms, PS-springer, SW-mz}. 
We summarize here the relevant result.

\begin{proposition}
\label{prop:resprod}
Let $A=B \otimes_{\k} C$ be the tensor product of two connected, 
finite-type graded $\k$-algebras. Then 
\begin{align*}
\RR^1_k(B \otimes_\k C)&=\RR^1_k(B)\times \{0\} \cup \{0\}\times \RR^1_k(C),\\
\RR^i_1(B \otimes_\k C)&=\bigcup\limits_{p\ge 0} \RR^p_1(B)\times  \RR^{i-p}_1(C), 
 \quad \textrm{ if } i\geq 2.
\end{align*}
\end{proposition}

\begin{proof}
As in \cite{PS-plms, SW-mz}, the claim easily follows from the 
following fact:  if $a = (b,c)$ is an element in $A^1=B^1\oplus C^1$, then 
the cochain complex $(A,a)$ splits as a tensor product of cochain 
complexes, $(B,b)\otimes (C,c)$, and thus 
$b_i(A,a)=\sum_{p+q=i} b_p(B,b)b_q(C,c)$. 
\end{proof}

\subsection{Naturality properties}
\label{subset:natural}

The resonance varieties enjoy several nice naturality 
properties with respect to morphisms of graded algebras.  
To describe some of these properties, we start with a Lemma/Definition, 
following the approach from \cite{DS18}, where a more general situation 
is studied.

\begin{lemma}
\label{lem:phistar}
Let $\varphi\colon A\to B$ be a morphism of graded $\k$-algebras. 
For each $a\in A^1$, there is an induced homomorphism 
\begin{equation}
\label{eq:phi_a}
\varphi_a\colon H^{\hdot}(A,\delta_a)\to 
H^{\hdot}(B,\delta_{\varphi(a)}).
\end{equation}
\end{lemma}

\begin{proof}
Let $[b]\in H^i(A,a)$, 
represented by an element $b\in A^i$ such that $ab=0$ in $A^{i+1}$.  
Since  $\varphi(a)\varphi(b)=0$, we may define a map 
$\varphi_a $ from $H^{\hdot}(A,\delta_a)$ to $H^{\hdot}(B,\delta_{\varphi(a)})$ 
by sending $[b]$ to $[\varphi(b)]$.  To verify this map is well-defined, 
suppose $b=ac$, for some $c\in A^{i-1}$; then $\varphi(b)=\varphi(a)\varphi(c)$, 
and so $[\varphi(b)]=[\varphi(c)]$. 
\end{proof}

\begin{proposition}
\label{prop:functres}
Let $\varphi\colon A\to B$ be a morphism of graded 
algebras such that $\varphi^i\colon A^i\to B^i$ is injective and 
$\varphi^{i-1}$ is surjective, for some $i\ge 1$. Then 
\begin{enumerate}
\item \label{mr1}
The homomorphisms $\varphi^i_a\colon H^{i}(A,\delta_a)\to 
H^{i}(B,\delta_{\varphi(a)})$ are injective, for all $a\in A^1$.
\item \label{mr2}
Suppose further that the map $\varphi^1\colon A^1\to B^1$ is 
injective.  Then this map restricts to inclusions 
$\varphi^1\colon \RR^i_k(A)\inj  \RR^i_k(B)$, for all $k\ge 0$.
\end{enumerate}
\end{proposition}

\begin{proof}
To prove part \eqref{mr1}, suppose that $\varphi^i_a ([b])=0$, 
for some $b\in A^i$.  Then 
$\varphi^i(b)=\varphi^1(a)v$, for some $v\in B^{i-1}$. By 
our surjectivity assumption on $\varphi^{i-1}$, there is an 
element $u\in A^{i-1}$ such that $\varphi^{i-1}(u)=v$, 
and so $\varphi^i(b)=\varphi^i(av)$.  Our injectivity 
assumption on $\varphi^{i}$ now implies that  $b=av$, and 
so $[b]=0$.

Part \eqref{mr2} follows at once from part \eqref{mr1} and the 
definition of resonance varieties.
\end{proof}

As a particular case, we recover a result from \cite{PS-mathann, Su-toulouse}.

\begin{corollary}
\label{cor:functorial resonance}
Let $\varphi\colon A\to B$ be a morphism of graded, connected 
algebras.  If the map  $\varphi^1\colon A^1 \to B^1$ is injective, then 
$\varphi^1(\RR^1_k(A))\subseteq  \RR^1_k(B)$, for all $k\ge 0$.
\end{corollary}

It follows that the resonance varieties of a graded, connected algebra $A$ depend 
only on the isomorphism type of $A$.  More precisely, if $\varphi \colon A\isom B$ 
is an isomorphism between two such algebras, then the linear isomorphism 
$\varphi^1\colon A^1\isom B^1$ restricts to isomorphisms 
$\varphi^1\colon \RR^i_k(A) \isom \RR^i_k(B)$ for all $k\ge 0$. 

In general, though, even if $\varphi\colon A\to B$ is  
an injective morphism between two graded algebras, the set 
$\varphi^1(\RR^i_k(A))$ may not be included in $\RR^i_k(B)$, 
for some $i>1$ and $k>0$.

\begin{example}
\label{ex:notinc}
Let $f\colon S^1\times S^1\to S^1\vee S^2$ be the map obtained 
(up to homotopy) by pinching a meridian circle of the torus to a point, 
and let $\varphi\colon A\to B$ be the induced morphism between 
the respective cohomology $\k$-algebras. It is readily seen that 
$\varphi$ is injective, yet $\RR^2_1(A)=\k$, whereas $\RR^2_1(B)=\{0\}$.
\end{example}

\section{Resonance and the BGG correspondence}
\label{sect:bgg}

In this section we explain how the BGG correspondence can be used 
to find equations for the resonance varieties of a graded algebra, and 
discuss the behavior of these varieties under coproducts, and under 
injective morphisms of algebras. 

\subsection{Equations for the resonance varieties}
\label{subsec:eqresvar}
Once again, let $A$ be a connected, finite-type cga  
over a field $\k$.  Without essential loss of generality, we will 
assume that $n:=b_1(A)$ is at least $1$. Let us pick a basis 
$\{ e_1,\dots, e_n \}$ for the $\k$-vector space 
$A^1$, and let $\{ x_1,\dots, x_n \}$ be the Kronecker 
dual basis for the dual vector space $A_1=(A^1)^*$.  
These choices allow us to identify the symmetric algebra $\Sym(A_1)$ 
with the polynomial ring $S=\k[x_1,\dots, x_n]$. 

The Bernstein--Bernstein--Gelfand correspondence (see for 
instance \cite[\S7B]{Ei}) yields a cochain complex of finitely generated,
 free $S$-modules, 
\begin{equation}
\label{eq:univ aomoto}
\xymatrixcolsep{20pt}
\bL(A)=(A\otimes_{\k} S,\delta)\colon 
\xymatrix{
\cdots \ar[r] 
&A^{i-1}\otimes_{\k} S \ar^(.55){\delta^{i-1}_A}[r] 
&A^{i} \otimes_{\k} S \ar^(.5){\delta^{i}_A}[r] 
&A^{i+1} \otimes_{\k} S \ar[r] 
& \cdots},
\end{equation}
with differentials given by $\delta^{i}_A(u \otimes 1)= \sum_{j=1}^{n} 
e_j u \otimes x_j$ for $u\in A^{i}$. By construction, the matrices associated 
to these differentials have entries that are linear forms in the variables of $S$.    

It is readily verified that the evaluation of the cochain complex $\bL(A)$ 
at an element $a\in A^1$ coincides with the cochain complex $(A,\delta_a)$ 
from \eqref{eq:aomoto}, that is to say, $\left.\delta^i_A\right|_{x_j=a_j} = \delta^i_a$. 
By definition, an element $a \in A^1$ belongs to $\RR^i_k(A)$ 
if and only if 
\begin{equation}
\label{eq:rank delta}
\rank \delta^{i-1}_a + \rank \delta^{i}_a \le b_i(A) -k,
\end{equation}
where recall $b_i(A)=\dim_{\k} A^i$. 
Let $I_r(\psi)$ denote the ideal of $r\times r$ minors 
of a $p\times q$ matrix $\psi$ with entries in $S$, with 
the convention that $I_0(\psi)=S$ and $I_r(\psi)=0$ if 
$r>\min(p,q)$.  Using the well-known fact that 
$I_r (\phi\oplus \psi)= \sum_{s+t=r} I_s(\phi ) \cdot  I_{t}(\psi)$,
we infer that 
\begin{equation}
\begin{aligned}
\label{eq:rika}
\RR^i_k(A)&= V \Big( I_{b_{i}(A)-k+1} \big(\delta^{i-1}_A\oplus \delta^{i}_A\big) \Big)\\ 
&=\bigcap_{s+t=b_{i}(A)-k+1} \Big( V\big(I_s(\delta^{i-1}_A)\big) 
\cup V\big(I_{t}(\delta^{i}_A)\big)\Big).
\end{aligned}
\end{equation}

The degree $1$ resonance varieties admit an even simpler description.  
Clearly, the map $\delta^0_A\colon S\to S^n$ has matrix $\big( x_1 \cdots x_n\big)$, 
and so $V\big(I_1(\delta^{0}_A)\big)=\{0\}$; hence, 
\begin{equation}
\label{eq:r1ka}
\RR^1_k(A)=  V ( I_{n-k} (\delta^1_A) )
\end{equation}
for $0\le k<n$ and $\RR^1_n(A)=\{0\}$. 

\begin{remark}
\label{rem:scheme}
It is sometimes useful to consider the {\em resonance schemes}\/ 
$\mathbfcal{R}^i_k(A)$ of a graded algebra $A$ as above.  These schemes 
are defined by the ideals $I_{b_{i}(A)-k+1} \big(\delta^{i-1}_A\oplus \delta^{i}_A\big)$ 
from \eqref{eq:rika}, and have as underlying 
sets the resonance varieties $\RR^i_k(A)$.
\end{remark}

\subsection{Induced morphisms in cohomology}
\label{subsec:ind}

Given an arbitrary morphism $\varphi\colon A\to B$ of connected, 
finite-type graded $\k$-algebras, it is not clear how to define an 
induced chain map, $\bL(\varphi)\colon \bL(A)\to \bL(B)$. 
Nevertheless, when $\varphi$ is injective, this can be 
done (after making some non-canonical choices), following the 
approach from \cite{DS18}. 

Since each map $\varphi^i\colon A^i\inj B^i$ is injective, 
the $\k$-dual map, $\varphi_i\colon B_i\surj A_i$, 
is surjective.   
Let $\psi_i\colon A_i\inj B_i$ be a $\k$-linear splitting of 
$\varphi_i$, so that $\varphi_i\circ \psi_i=\id_{A_i}$.

\begin{lemma}
The map of $S$-modules $\bL(\varphi)\colon \bL(A)\to \bL(B)$ defined by 
\begin{equation}
\label{eq:chain bgg}
\begin{gathered}
\xymatrixcolsep{20pt}
\xymatrix{
 \bL(A):\ar^{\bL(\varphi)}[d]  \hspace*{-20pt} 
 &A^0 \otimes_{\k} \Sym(A_1)  \ar[r]^(.5){\delta^0_A} 
\ar^{\varphi^{0}\otimes \Sym(\psi_1)}[d]
&  A^1 \otimes_{\k} \Sym(A_1)   \ar[r]^(.5){\delta^1_A}   
\ar^{\varphi^{1}\otimes \Sym(\psi_1)}[d]
&  A^2 \otimes_{\k} \Sym(A_1) \ar[r]    
\ar^{\varphi^{2}\otimes \Sym(\psi_1)}[d]
& \cdots  \phantom{.}\\
 \bL(B) : \hspace*{-20pt} &B^0 \otimes_{\k} \Sym(B_1)  \ar[r]^(.5){\delta^0_B} 
&  B^1 \otimes_{\k} \Sym(B_1)   \ar[r]^(.5){\delta^1_B}  
&  B^2 \otimes_{\k} \Sym(B_1)  \ar[r] 
&  \cdots 
}
\end{gathered}
\end{equation}
is a chain map.
\end{lemma}
 
\begin{proof}
Pick bases 
$\{ e_1,\dots, e_n \}$ for $A^1$ and $\{ f_1,\dots, f_p\}$ for $B^1$ 
so that $\varphi^1(e_j)=f_j$ for $j\le p$ and $\varphi^1(e_j)=0$, 
otherwise.  Letting $\{ x_1,\dots, x_n \}$ and $\{ y_1,\dots, y_p\}$ 
be the dual bases for $A_1$ and $B_1$, respectively, we find that
\begin{align*}
(\varphi^{i+1}  \otimes \Sym(\psi_1) )\circ \delta^{i}_A (u\otimes 1)
&= \varphi^{i+1}  \otimes \Sym(\psi_1) \bigg( \sum_{j=1}^{n} 
e_j u \otimes x_j \bigg) \\
&=  \sum_{j=1}^{n}  \varphi^1(e_j) \varphi^i(u)  \otimes \psi_1(x_j) \\
&=  \sum_{j=1}^{p}  f_j \varphi^{i} (u)  \otimes y_j \\
&= \delta^{i}_B ( \varphi^{i} (u) \otimes 1)  \\
&= \delta^{i}_B  \circ (\varphi^{i}  \otimes \Sym(\psi_1)) (u\otimes 1),
\end{align*}
thus verifying our claim.  
\end{proof}

The chain map defined above induces a morphism 
in cohomology, $\bL(\varphi)^*\colon H^{\hdot}(\bL(A))\to H^{\hdot}(\bL(B))$. 
The next proposition follows at once.

\begin{proposition}
For each $i\ge 0$, the evaluation of the morphism 
$\bL(\varphi)^*\colon H^{i}(\bL(A))\to H^{i}(\bL(B))$
at a point $a\in A^1$ yields the map 
$\varphi^i_a\colon H^i(A,\delta_a)\to H^i(B,\delta_{\varphi(a)})$ 
from \eqref{eq:phi_a}.
\end{proposition}

\subsection{Resonance varieties of coproducts}
\label{subsec:rescoprod}

Let $B$ and $C$ be 
two connected cga's.  Their wedge sum, $B\vee C$, is a 
new connected cga, whose underlying graded vector space in 
positive degrees is $B^+\oplus C^+$, with multiplication 
$(b,c)\cdot (b',c') = (bb', cc')$.  
The next proposition sharpens results from 
\cite{PS-plms, SW-mz}.  Since this is a new proof, and since we 
will use the same approach to prove Theorem \ref{thm:nonmax} below, 
we give complete details. 
 
\begin{proposition}
\label{prop:rescoprod}
Let $A=B \vee C$ be the wedge sum of two connected, finite-type 
graded $\k$-algebras with $b_1(B)>0$ and $b_1(C)>0$. 
Identifying $A^1=B^1\oplus C^1$, we have
\[
\RR^i_k(A)=
\begin{cases}
\ \bigcup\limits_{s+t=k-1} \RR^1_s(B) \times \RR^1_t(C)
&\quad \textrm{if $i=1$}, 
\\[3pt]
\hspace{6pt} \bigcup\limits_{s+t=k}\hspace{6pt} 
\RR^i_s(B)\times  \RR^i_t(C)
&\quad \textrm{if $i\ge 2$}.
\end{cases}
\]
\end{proposition}
\begin{proof}
Note that $\bL(A)^+ = \bL(B)^+ \oplus \bL(C)^+$.  
Thus, for $i>0$ the matrix of $\delta_A^i$ is the block sum of the matrices 
of $\delta_B^i$ and $\delta_C^i$, and so 
$I_r (\delta_A^i )= \sum_{s+t=r} I_s(\delta_B^i ) \cdot  I_{t}(\delta_C^i)$, 
where $I_s(\delta_B^i )$ and $I_t(\delta_C^i)$ are viewed as ideals 
of $S=\Sym(A_1)$ by extension of scalars.  When $i=1$, we get 
\begin{align*}
\RR^1_k(A)&=V \big( I_{b_{1}(A)-k} (\delta^{1}_A) \big)
\\
&=V \big( I_{b_{1}(A)-k} (\delta^{1}_B\oplus \delta^{1}_C \big)
\\
&=V \Big( \sum_{s+t=b_{1}(A)-k}   
I_{s} (\delta^{1}_B) \cdot I_t( \delta^{1}_C) \Big)
\\
&=\bigcap_{s+t=b_{1}(A)-k} \Big( V\big(I_{s} (\delta^{1}_B )\big) \cup 
V\big(I_{t}(\delta^{1}_C) \big)\Big)
\\
&=\bigcap_{u+v=k} \Big( V\big(I_{b_1(B)-u} (\delta^{1}_B )\big) \cup 
V\big(I_{b_1(C)-v}(\delta^{1}_C) \big)\Big)
\\
& = \bigcap_{u+v=k} \Big(\big(\RR^1_u(B)\times C^1\big) \cup 
\big(B^1\times \RR^1_v(C)\big)\Big)\\
&= \bigcup\limits_{s+t=k-1} \RR^1_s(B) \times \RR^1_t(C)
\end{align*}
where the last step is set-theoretical, based solely on the resonance 
filtrations \eqref{eq:resfilt} for the algebras $B$ and $C$. 
The proof for the case $i>1$ is similar. 
\end{proof}

\section{Poincar\'{e} duality algebras and alternating forms}
\label{sect:pd-res}

In this section we consider  a restricted class of graded algebras 
which abstract the notion of Poincar\'{e} duality for closed, oriented 
topological manifolds, and we discuss the alternating form naturally 
associated to such an algebra. 

\subsection{Poincar\'{e} duality}
\label{subsec:def pd}
Let $A$ be a non-negatively graded, 
graded-commutative algebra over a field $\k$.  
We will assume throughout that $A$ is connected 
and locally finite.  We say that $A$ is a 
{\em Poincar\'{e} duality $\k$-algebra}\/ of 
formal dimension $m$ if there is a $\k$-linear 
map $\varepsilon\colon A^m \to \k$ (called an 
{\em orientation}) such that all the bilinear forms 
\begin{equation}
\label{eq:duality}
A^i \otimes_{\k} A^{m-i}\to \k, \quad a\otimes b\mapsto \varepsilon (ab)
\end{equation}
are non-singular.  It follows  $\varepsilon$ is an isomorphism, 
and that $A^i=0$ for $i>m$.  Furthermore, for each $0\le i \le m$, 
there is an isomorphism
\begin{equation}
\label{eq:pd}
\PD^i\colon A^{i}\to (A^{m-i})^*, \quad  \PD^i(a)(b)=\varepsilon (ab).
\end{equation}

Consequently, each element $a\in A^i$ has a ``Poincar\'{e} 
dual," $a^{\vee}\in A^{m-i}$, which is uniquely determined by the 
formula $\varepsilon (a a^{\vee})=1$. 
We define the orientation class $\omega_A\in A^m$ as the 
Poincar\'{e} dual of $1\in A^0$, that is, $\omega_A=1^{\vee}$. 
Conversely, a choice of orientation 
class $\omega_A\in A^m$ defines an orientation 
$\varepsilon\colon A^m \to \k$ by setting $\varepsilon(\omega_A)=1$.

In more algebraic terms, a $\PD_m$ algebra is a graded, 
graded-commutative Gorenstein Artin algebra of socle degree $m$.

The main motivation for these definitions comes from topology: 
if $M$ is a compact, connected, orientable, $m$-dimensional 
manifold, then, by Poincar\'{e} duality, the cohomology algebra 
$A=H^{\hdot}(M,\k)$ is a $\PD_m$ algebra over $\k$, with the 
orientation class $[M]\in H_m(M,\k)$ determining the orientation 
on $A$ by setting $\omega_A([M])=1$.

\subsection{Tensor products and connected sums}
\label{subsec:prod-sums}
The class of Poincar\'{e} duality algebras is closed under 
taking tensor products and connected sums.  

Indeed, if $A$ and $B$ are Poincar\'{e} duality algebras of 
dimension $m$ and $n$, respectively, then their tensor product, 
$A\otimes_\k B$, is a Poincar\'{e} duality algebra of dimension 
$m+n$.   Conversely, if the tensor product of two graded algebras is 
a $\PD$ algebra, then each factor must be a $\PD$ algebra,   
see for instance \cite[p. 188]{Hal} or \cite[Prop.~3.3]{SS10}.  

Now let $A$ and $B$ be two $\PD_m$ algebras, with orientation 
classes $\omega_A$ and $\omega_B$, respectively.  Much as in 
\cite{MeS}, let us define their {\em connected sum}, $C=A\# B$, 
as the pushout 
\begin{equation}
\label{eq:pushout}
\begin{gathered}
\xymatrix{
\bwedge (\omega) \ar^-{\omega\mapsto \omega_A}[r] 
\ar[d]_{\substack{\omega\\\downmapsto\\\omega_B}} 
& A\ar[d] 
\\
B \ar[r] 
& A \# B
}
\end{gathered}
\end{equation}

In other words, $C^0=\k\cdot 1$, 
$C^i=A^i \oplus B^i$ for $0<i<m$, and $C^m=\k \cdot \omega_C$, 
with $\omega_A$ and $\omega_B$ identified to $\omega_C$, 
and with multiplication defined in the obvious way. 

The motivation and terminology for the above notions comes 
from manifold topology.  Indeed, if $M$ and $N$ are two closed, 
oriented manifolds, then $M\times N$ is again a closed, oriented 
manifold, and $H^{\hdot}(M\times N,\k)\cong 
H^{\hdot}(M,\k) \otimes_{\k} H^{\hdot}(N,\k)$. Moreover, the 
cohomology algebra of the connected sum of two closed, 
oriented manifolds of the same dimension is the connected 
sum of the respective cohomology algebras, that is, 
$H^{\hdot}(M\# N,\k)\cong H^{\hdot}(M,\k) \#H^{\hdot}(N,\k)$. 

\subsection{The alternating form of a $\PD_m$ algebra}
\label{subsec:alt-forms}
Associated to a $\PD_m$ algebra over a field 
$\k$ there is an alternating $m$-form,
\begin{equation}
\label{eq:mu form}
\mu_A\colon \bwedge^m A^1 \to \k, \quad
\mu_A(a_1\wedge \cdots \wedge a_m)= \varepsilon (a_1\cdots a_m).
\end{equation} 

Let us specialize now to the case when $m=3$. 
In this instance, the multiplicative structure of $A$ can 
be recovered from the $3$-form $\mu=\mu_A$ 
and the orientation $\varepsilon$, as follows. 
As before, set $n=b_1(A)$, and fix a basis $\{e_1,\dots ,e_n\}$ for $A^1$.  
Let $\{e_1^{\vee},\dots ,e_n^{\vee}\}$ be the Poincar\'{e} dual basis for $A^2$, 
and take as generator of $A^3=\k$ the class $\omega=1^{\vee}$.   
The multiplication in $A$, then, is given on basis elements by 
\begin{equation}
\label{eq:mult}
e_i  e_j=\sum_{k=1}^{n} \mu_{ijk}\,  e^{\vee}_k, \quad 
e_i e_j^{\vee} = \delta_{ij} \omega,
\end{equation} 
where $\mu_{ijk}=\mu(e_i\wedge e_j\wedge e_k )$ and $\delta_{ij}$ 
is the Kronecker delta.   An alternate way to encode this information 
is to let $A_i=(A^i)^{*}$ be the dual $\k$-vector space and to let $e^i\in A_1$ 
be the (Kronecker) dual of $e_i$.  We may then view $\mu=\mu_A$ dually 
as a trivector, 
\begin{equation}
\label{eq:trivec}
\mu =\sum \mu_{ijk}\, e^i \wedge e^j \wedge e^k \in \bwedge^3 A_1,
\end{equation} 
and will sometimes abbreviate this as $\mu =\sum \mu_{ijk}\, e^i  e^j e^k$. 

\begin{example}
\label{ex:trisum}
It is readily seen that the trivector associated 
to a connected sum of two $\PD_3$ algebras is the sum of the 
corresponding trivectors; that is, 
\begin{equation}
\label{eq:trivec sum}
\mu_{A\#B} = \mu_A+\mu_B.
\end{equation}
\end{example}

Any alternating $3$-form $\mu \colon \bwedge^3 V \to \k$ 
on a finite-dimensional $\k$-vector space $V$ determines a $\PD_3$ 
algebra $A$ over $\k$ for which $\mu_A=\mu$:  simply take $A^0=A^3=\k$ 
and $A^1=A^2=V$, choose dual bases as above, and define the 
multiplication map as in \eqref{eq:mult}.

\begin{remark}
\label{rem:roos}
In \cite{Rs}, Roos outlined procedures for writing down a presentation 
for the algebra $A$ in terms of the trivector $\mu$, and for determining 
whether $A$ is a Koszul algebra.
\end{remark}

\begin{remark}
\label{rem:sullivan}
In \cite{Su75}, Sullivan showed that every alternating $3$-form over 
a field $\k$ of characteristic $0$ can be realized as the $3$-form  
associated to the cohomology algebra $A=H^{\hdot}(M,\k)$ of a 
closed, oriented $3$-manifold $M$. 
\end{remark}

\subsection{Classification of alternating forms}
\label{subsec:trivectors}
Let $V$ be a $\k$-vector space of dimension $n$, and let 
$\bwedge^m(V^*)$ be the vector space of alternating 
$m$-forms on $V$. The general linear group $\GL(V)$ 
acts on this affine space by 
\begin{equation}
\label{eq:gact}
(g\cdot \mu) (a_1\wedge \dots \wedge a_m) := 
\mu\, \big(g^{-1} a_1 \wedge \dots \wedge g^{-1} a_m\big).
\end{equation}

The orbits of this action are the equivalence classes of alternating 
$m$-forms on $V$. (We write $\mu\sim \mu'$ if $\mu'=g\cdot \mu$.) 
Over $\overline{\k}$, the Zariski closures of these orbits 
define affine algebraic varieties.  A standard dimension argument with 
algebraic groups (see e.g.~\cite{CH}) shows that there can be finitely 
many orbits over $\overline{\k}$ only if $n^2\ge \binom{n}{m}$, that is, 
$m\le 2$ or $m=3$ and $n\le 8$.  Furthermore, when $\k=\R$ and $\overline{\k}=\C$, 
each complex orbit has only finitely many real forms, by \cite[Prop.~2.3]{BH}. 

Let us specialize now to the case of most interest to us, to wit, $m=3$. 
For $\k=\C$, the classification of alternating trilinear forms was carried 
out by Schouten \cite{Sc} in dimensions $n \le 7$ and by Gurevich \cite{Gu} 
for $n=8$.  For $\k=\R$, the classification was done by Gurevich, Revoy, 
and Westwick for $n\le 7$ and by Djokovi\'{c} \cite{Dj1} for $n= 8$. 
The classification in dimensions $n\le 7$ was extended to arbitrary 
fields by Cohen and Helminck \cite{CH}; 

Over $\C$ there are $23$ orbits in dimension $n = 8$.  Lying in 
the closure of another orbit defines a partial order on the set of 
orbits; the corresponding Hasse diagram is given in \cite{Dj2}.  
Those $23$ complex orbits split into either $1$, $2$, or $3$ 
real orbits, for a total of $35$ orbits, as indicated in \cite{Dj1}.  
Representative trivectors for each one of these $\C$-orbits 
(and the corresponding $\R$-orbits for $n\le 7$) are given 
in the tables from Appendix \ref{sect:tables}.

\subsection{Maps of non-zero degree}
\label{subsec:deg1}

Let $A$ and $B$ be two $\PD_m$ algebras.  We say that a 
morphism of graded algebras $\varphi\colon A\to B$ has 
{\em non-zero degree}\/ if the linear map $\varphi^m\colon A^m \to B^m$ 
is non-zero.  In this case, we may pick orientation classes 
such that 
\begin{equation}
\label{eq:om}
\varphi^m(\omega_A)= \omega_B.
\end{equation}
Consequently, $\varphi$ is compatible with the Poincar\'{e} duality 
isomorphisms from \eqref{eq:pd}, that is, $(\varphi^{m-i})^*\circ \PD^i_A=
\PD^i_B\circ \varphi^i$, for $0\le i\le m$.  It follows that 
\begin{equation}
\label{eq:mu phi}
\mu_B \circ \bwedge^m \varphi^1=\mu_A.
\end{equation}

Once again, the terminology comes from topology: if $f\colon M\to N$ 
is a map of degree $d\ne 0$ between two closed, oriented manifolds of 
dimension $m$, then the induced morphism in cohomology, 
$f^*\colon H^{\hdot}(N,\k)\to H^{\hdot}(M,\k)$ will restrict to 
multiplication by $d$ in degree $m$.  Thus, if the characteristic  
of $\k$ does not divide $d$ (for instance, if $\ch \k=0$), then 
the morphism $f^*$ has non-zero degree.

We shall need the following alternate way to express 
the naturality of Poincar\'{e} duality with respect to 
non-zero degree morphisms (compare with 
\cite[Lemma I.3.1]{MeS}). 

\begin{lemma}
\label{lem:deg1}
Let $\varphi\colon A\to B$ be a non-zero degree morphism 
between two $\PD_m$ algebras. Then 
$\varphi(a^{\vee}) = \varphi(a)^{\vee}$, for all 
homogeneous elements $a\in A$.
\end{lemma}

\begin{proof}
We have $\varphi(a)\cdot \varphi(a^{\vee}) = \varphi(aa^{\vee})= 
\varphi(\omega_A) = \omega_{B}$, and the claim follows at once.
\end{proof}

\begin{proposition}
\label{prop:deg1-inj}
A morphism $\varphi\colon A\to B$ 
between two $\PD_m$ algebras is injective if and only if 
$\varphi$ has non-zero degree. 
\end{proposition}

\begin{proof}
If $\varphi$ is injective, then in particular $\varphi^m$ is injective, 
and thus is non-zero. For the converse, suppose $\varphi$ has non-zero 
degree. By the proof of the above lemma, $\varphi(a)\ne 0$, for all 
homogeneous elements $a\in A$, and the claim follows. 
\end{proof}

For instance, if $A=B\#C$, then the canonical morphisms $B\to A$ and 
$B\to C$ are injective, and thus have non-zero degree. 

An isomorphism of $\PD_m$ algebras is a map $\varphi\colon A\to B$ 
between two $\PD_m$ algebras which preserves both the graded algebra 
structures and the orientation classes. 

\begin{proposition}
\label{prop:pdm-iso}
Two $\PD_m$ algebras $A$ and $B$ are isomorphic as $\PD_m$ algebras 
if and only if they are isomorphic as graded algebras. 
Furthermore, either of these conditions implies that $\mu_A \sim \mu_B$.
\end{proposition} 

\begin{proof} 
By Proposition \ref{prop:deg1-inj}, if $\varphi\colon A\to B$ is 
an isomorphism between the two underlying graded algebras, 
then condition \eqref{eq:om} is satisfied, and so $\varphi$ is an isomorphism of 
$\PD_m$ algebras. The converse is obvious.

Suppose now that $\varphi\colon A\to B$ is 
an isomorphism of $\PD_m$ algebras.  Then, by \eqref{eq:mu phi}, we 
have that $\mu_B \circ \bwedge^m \varphi^1=\mu_A$, that is, 
$\mu_B=\varphi^1\cdot \mu_A$, and so $\mu_A \sim \mu_B$.
\end{proof}

\begin{theorem}
\label{thm:pd3-iso}
For two $\PD_3$ algebras $A$ and $B$, the 
following are equivalent.
\begin{enumerate}
\item \label{pd3-1} $A\cong B$, as $\PD_m$ algebras. 
\item \label{pd3-2} $A\cong B$, as graded algebras. 
\item \label{pd3-3} $\mu_A \sim \mu_B$. 
\end{enumerate}
\end{theorem} 

\begin{proof}
In view of Proposition \ref{prop:pdm-iso}, we only need to show that 
\eqref{pd3-3} $\Rightarrow$ \eqref{pd3-2}. 
Suppose $\mu_A \sim \mu_B$.  There is then a linear 
isomorphism $g\colon A^1\to B^1$ such that  $\mu_B=g\cdot \mu_A$, 
that is, $\omega_B=(\bigwedge^3 g)(\omega_A)$.   Define a map 
$\varphi\colon A\to B$ by requiring $\varphi^0=\id$, $\varphi^1=g$, 
$\varphi^2=g^{\vee}$, and $\varphi^3=\bigwedge^3 g$, 
where $g^{\vee}\colon A^2\to B^2$ is given by 
$g^{\vee}(a^{\vee})=(g(a))^{\vee}$.  Clearly, $\varphi$ is also a 
linear isomorphism. Now let $a,b\in A^1$ be two non-zero elements. 
Setting $c=(ab)^{\vee}$, we have 
\[
\omega_B=\varphi^3(g)(\omega_A)=\varphi^3(g)(abc)=g(a)g(b)g(c),
\]
and so 
\[
\varphi(ab)=g^{\vee}(ab)=g^{\vee}(c^{\vee})=g(c)^{\vee}=g(a)g(b)=\varphi(a)\varphi(b).
\]
It follows that $\varphi$ is an isomorphism of graded algebras, and we are done.
\end{proof}

In conclusion, the constructions from \S\ref{subsec:alt-forms} together with the above 
theorem establish a one-to-one correspondence between isomorphism 
classes of $3$-dimensional Poincar\'e duality algebras and equivalence 
classes of alternating $3$-forms, given by $A \leftrightsquigarrow \mu_A$.

\section{Poincar\'{e} duality and resonance}
\label{sect:res pdm}

In this section we explore some of the constraints imposed  by 
Poincar\'{e} duality on the resonance varieties of a $\PD$ algebra. 
Henceforth, the ground field $\k$ will be assumed 
to be of characteristic different from $2$.

\subsection{Resonance varieties of $\PD_m$ algebras}
\label{subsec:pdres}

We start with a lemma expressing the compatibility between Poincar\'{e} 
duality and the BGG correspondence. A similar statement is proved in 
\cite[Lemma 7.3]{PS-sasaki}, in a more general context. For completeness, 
we provide a short proof.  

\begin{lemma}
\label{lem:uptosign}
Let $A$ be a $\PD_m$ algebra.  Then, for 
all $0\le i\le m$ and all $a\in A^1$, we have 
a commuting square,
\[
\xymatrixcolsep{42pt}
\xymatrix{
(A^{m-i})^*   \ar[r]^{(\delta^{m-i-1}_{-a})^*} 
& (A^{m-i-1})^*  \\
 A^{i}   \ar[r]^{\delta^i_{a}} \ar_{\Phi_i}^{\cong}[u] 
 &A^{i+1}\, , \ar_{\Phi_{i+1}}^{\cong}[u]
}
\]
where $\Phi_i=(-1)^{i} \PD^i$ . 
\end{lemma}

\begin{proof}
Let $b \in A^{i}$ and $c \in A^{m-i-1}$.  Then
$
\PD\circ \delta_{a}(b) (c)  =\PD( a b)  (c) =\varepsilon(abc),
$
while $\delta^*_{-a}\circ \PD (b) (c)=
\PD (b) ( \delta_{-a} (c) )  =-\PD(b) (a  c)=-\varepsilon(bac).
$
Since $ab=(-1)^{i}ba$, we are done.
\end{proof}

The next corollary follows at once.

\begin{corollary}
\label{cor:twpd}
Let $A$ be a $\PD_m$ algebra.  Then, for all $0\le i\le m$ 
and all $a\in A^1$, 
\[
\left(H^i(A,\delta_{a})\right)^* \cong H^{m-i}(A,\delta_{-a}).
\] 
Furthermore, if $\varphi\colon A\to B$ is a 
morphism between two $\PD_m$ algebras, then the map 
$\varphi^i_a\colon H^i(A,\delta_a)\to H^i(B,\delta_{\varphi(a)})$  
from \eqref{eq:phi_a} is dual to  
$\varphi^{m-i}_{-a}\colon H^{m-i}(A,\delta_{-a})\to 
H^{m-i}(B,\delta_{-\varphi(a)})$. 
\end{corollary}

We are now ready to state and prove the resonance analogue of the 
palindromicity of the Betti numbers of a Poincar\'{e} duality algebra.

\begin{theorem}
\label{thm:rpdf}
Let $A$ be a $\textup{PD}_m$-algebra.  Then, for all $i$ and $k$, 
\[
\RR^{i}_k(A)=\RR^{m-i}_k(A).
\] 
\end{theorem}

\begin{proof}
By Corollary \ref{cor:twpd},  the $\k$-vector space $H^i(A,\delta_{a})$ 
is dual to $H^{m-i}(A,\delta_{-a})$. The claimed equality follows straight 
from the definition of resonance.
\end{proof}

This theorem shows that it is enough to compute the resonance  
varieties of a $\PD_m$ algebra in degrees up to the middle dimension: 
the other ones are then essentially given by Poincar\'{e} duality. 

As a consequence of Theorem \ref{thm:rpdf}, we deduce that $\RR^{m}_1(A)=\{0\}$, 
a fact which was proved in a somewhat different fashion in \cite[Prop.~5.14]{DSY}. 
Moreover, in view of formula \eqref{eq:bia}, we recover the fact that 
$b_i(A)=b_{m-i}(A)$.  Thus, the above theorem may be regarded as 
a generalization of the palindromicity of the Poincar\'{e} polynomial 
of a closed, orientable manifold.  

\subsection{Connected sums and resonance}
\label{subsec:sumres}
The resonance varieties of a connected sum of two Poincar\'{e} duality algebras 
can be computed in terms of the resonance varieties of the factors. 
Arguing as in the proof of Proposition \ref{prop:rescoprod}, we 
obtain the following result.

\begin{proposition}
\label{prop:resconnsum}
Let $A=B\# C$ be the connected sum of two $\PD_m$ algebras 
with positive first Betti numbers. Then, for all $k\ge 0$, 
\begin{equation}
\label{eq:r1connsum}
 \RR_k^i(A)=
\begin{cases}
\ \bigcup\limits_{s+t=k-1} \RR^i_s(B) \times \RR^i_t(C)
&\quad \textrm{if $i=1$ or $m-1$},
\\[3pt]
\hspace{8pt} \bigcup\limits_{s+t=k}\hspace{6pt} \RR^i_s(B)\times  \RR^i_t(C)
&\quad \textrm{if $1<i <m-1$},
\\[3pt]
\hspace{11pt} \{0\}
&\quad \textrm{if $i=0$ or $m$, and $k=1$},
\end{cases}
\end{equation}
and $\RR_k^i(A)=\emptyset$, otherwise.
\end{proposition}

\begin{corollary}
\label{cor:resb1}
Let $A=B\# C$ be the connected sum of two $\PD_m$ algebras. 
If $b_1(B)>0$ and $b_1(C)>0$, then $\RR^1_1(A)=A^1$.
\end{corollary}

\begin{example}
\label{ex:surfaces}
Let $A=H^{\hdot}(\Sigma_g,\k)$  be the cohomology algebra 
of a closed, orientable surface of genus $g\ge 2$.  Since 
$\Sigma_g \cong \Sigma_{g-1}\# S^1\times S^1$, 
the above corollary yields $\RR^1_1(A)=A^1$. 
\end{example}

\subsection{A resonance obstruction to domination}
\label{subsec:obstruction}

A fundamental question in manifold topology (studied by Gromov \cite{Gr} 
and others) is to decide whether there exists a map $f\colon M\to N$ 
of non-zero degree between two closed, oriented manifolds $M$ and 
$N$ of the same dimension.  If such a map exists, one says that $M$ 
{\em dominates}\/ $N$. 

By analogy, given two $\PD_m$ algebras $A$ and $B$, we say that 
$B$ dominates $A$ if there is a non-zero degree morphism 
$A\to B$.  By Proposition \ref{prop:deg1-inj}, this 
is equivalent to saying there is an injective morphism $A\to B$; 
in particular, we must have $b_i(A) \le b_i(B)$ for all $i\ge 0$. 
Applying Corollary \ref{cor:functorial resonance}, we obtain a 
geometric obstruction to domination.

\begin{corollary}
\label{cor:ff}
Suppose $\RR^1_k(A)$ has larger dimension (or more 
irreducible components) than $\RR^1_k(B)$, for some $k\ge 1$. 
Then $B$ does not dominate $A$.
\end{corollary}

\begin{example}
\label{ex:extalg}
The exterior algebra $E=\bigwedge (\k^m)$ is a Poincar\'{e} duality 
algebra of dimension $m$.  Since the Koszul complex $\bL(E)=E\otimes_{\k} S$ 
is exact, the resonance varieties of $E$ vanish; more precisely, 
$\RR^i_k(E) = \{0\}$ if $1\le k\le \binom{m}{i}$ and is empty, otherwise. 
It follows that $E$ does not dominate any $\PD_m$ algebra $A$ 
for which $\RR^1_1(A)$ has positive dimension. 
\end{example}

\section{The resonance varieties of a \texorpdfstring{$\PD_3$}{PD3} algebra}
\label{sect:res pd3}

We analyze now in more detail the structural 
properties of the resonance varieties of a $3$-dimensional 
Poincar\'{e} duality algebra. 

\subsection{Reduction to degree $1$ resonance}
\label{subsec:red1}
 The next proposition 
reduces the computation of the resonance varieties of a $\PD_3$ 
algebra to those in degree $1$. 

\begin{proposition}
\label{prop:r1r2}
Let $A$ be a $\PD_3$ algebra with $b_1(A)=n$.  Then 
\begin{enumerate}
\item \label{rr0}
$\RR^{i}_0(A)=A^1$.
\item \label{rr1}
$\RR^{3}_1(A)=\RR^0_1(A)=\{0\}$ and $\RR^{2}_n(A)=\RR^1_n(A)=\{0\}$.
\item \label{rr2}
 $\RR^{2}_k(A)=\RR^1_k(A)$ 
for $0< k < n$.
\item \label{rr3}
In all other cases, $\RR^i_k(A)=\emptyset$. 
\end{enumerate}
\end{proposition}

\begin{proof}
Statements  \eqref{rr0}, \eqref{rr1}, and \eqref{rr3} follow straight 
from the definitions and previous remarks, while 
\eqref{rr2} follows from Theorem \ref{thm:rpdf}.
\end{proof}

Thus, in order to understand the resonance varieties of a $\PD_3$ algebra 
$A$, it suffices to describe the resonance varieties $\RR^1_k(A)$, 
in depths $0< k< b_1(A)$. As a trivial example, suppose $\mu_A=0$;  then 
$\RR^1_k(A)=A^1$ for $k<b_1(A)$.

\subsection{Decomposable and irreducible forms}
\label{subsec:decomp-forms}
The next result further reduces the computation of the resonance 
varieties of an arbitrary $\PD_3$ algebra to those of a $\PD_3$ 
algebra whose associated $3$-form is irreducible. 

Let $\mu\colon \bwedge^3 V\to \k$ be an alternating $3$-form 
on a finite-dimensional $\k$-vector space $V$.   
The {\em rank}\/ of $\mu$ is the minimum 
dimension of a linear subspace $W\subset V$ such that $\mu$ factors
through $\bigwedge^3 W$; we write $\corank \mu=\dim V-\rank \mu$.
The $3$-form $\mu$ is said to be {\em irreducible}\/ if it 
has maximal rank, that is, $\corank \mu=0$.  

\begin{theorem}
\label{thm:nonmax}
Every $\PD_3$ algebra $A$ decomposes as 
$A\cong B\#C$, where $B$ are $C$ are $\PD_3$ 
algebras such that $\mu_B$ is irreducible and has 
the same rank as $\mu_A$, and $\mu_{C}=0$. 
Furthermore, the isomorphism $A^1\cong B^1\oplus C^1$ 
restricts to isomorphisms 
\begin{equation}
\label{eq:r1bar}
\RR^1_k(A) \cong  \RR^1_{k-r+1}(B) \times C^1
\cup \RR^1_{k-r}(B) \times \{0\}
\end{equation}
for all $k\ge 0$, where $r=\corank \mu_A$. 
\end{theorem}

\begin{proof}
Let $W\subset A^1$ be 
a subspace of dimension equal to $\rank \mu_A$ for 
which the form $\mu_A\colon \bigwedge^3 V\to \k$ factors through 
$\bigwedge^3 W$, and let $\bar{\mu}$ be the restriction of $\mu$ to $W$. 
By construction, this is a $3$-form whose rank equals that of $\mu$, that is, 
$\rank \bar\mu=\dim W$.  

Let $B$ be the $\PD_3$ algebra corresponding to $\bar{\mu}$.  Evidently, 
$B^1=W$ and $\mu_{B}=\bar\mu$ is irreducible.  It is now readily seen that 
$A\cong B\#C$, where $C$ is the $\PD_3$ algebra with $C^1=A^1/B^1$ 
and $\mu_{C}=0$.  

By a previous observation, $\RR^1_t( C)= C^1$ for $t<r$ and 
$\RR^1_r(C)=\{0\}$. Formula \eqref{eq:r1bar} now follows 
from Proposition \ref{prop:resconnsum}.
\end{proof}

\begin{remark}
\label{rem:decomp}
Suppose $A=H^{\hdot}(M,\k)$ is the cohomology algebra of a closed, orientable 
$3$-manifold $M$.  Write $M=N\#P$, where $P$ is the connected sum 
of the factors in the prime decomposition of $M$  having the $\k$-homology 
of either $S^3$ or $S^1 \times S^2$ 
and $N$ is the connected sum of all the other factors. Setting  
$B=H^{\hdot}(N,\k)$ and $C=H^{\hdot}(P,\k)$, we recover the 
decomposition $A\cong B\#C$ from the above result.
\end{remark}

As an immediate consequence of Theorem \ref{thm:nonmax}, we have 
the following corollary.

\begin{corollary}
\label{cor:res-corank}
If $A$ is a $\PD_3$ algebra, then $\RR^1_k(A)=A^1$ for 
all $k< \corank \mu_A$.
\end{corollary}

\subsection{Nullity and isotropic subspaces}
\label{subsec:res 3form}
Before proceeding, we need a few more classical definitions, 
suitably adapted to our setup (see for instance \cite{DrS14,Si}).  

Let $\mu\colon \bwedge^3 V\to \k$ be a $3$-form. 
A linear subspace $U\subset V$ is {\em $2$-singular}\/ 
with respect to $\mu$ 
if $\mu(a\wedge b\wedge c)=0$ for all $a,b\in U$ and $c\in V$. 
(If $\dim U=2$, we simply say $U$ is a {\em singular plane}.) 
The {\em nullity}\/ of $\mu$, denoted $\nl(\mu)$, is the maximum dimension 
of a $2$-singular subspace $U\subset V$.  Clearly, $V$ contains a 
$\mu$-singular plane if and only if $\nl(\mu)\ge 2$.  

The following (very simple) lemma clarifies the relationship between 
singularity and isotropicity in the context of $\PD_3$ algebras.

\begin{lemma}
\label{lem:sing}
Let $A$ be a $\PD_3$ algebra.  A linear subspace $U\subset A^1$ is 
$2$-singular (with respect to $\mu_A$) if and only if $U$ is isotropic. 
\end{lemma}

\begin{proof}
If $U\subset A^1$ is a $2$-singular subspace, then
$\mu_A(a\wedge b\wedge c)=\varepsilon(abc)=0$ for all $a,b\in U$ 
and $c\in A^1$. Since the bilinear form $A^2\otimes_\k A^1\to \k$, 
$\gamma\otimes c\mapsto \varepsilon (\gamma c)$ is non-degenerate, 
this implies $ab=0$ for all $a,b\in U$, that is, $U$ is isotropic.  

Conversely, if $U\subset A^1$ is an isotropic subspace, 
then $ab=0$ for all $a,b\in U$. Thus,  
$\mu_A(a\wedge b\wedge c)=\varepsilon(abc)=0$ for all $a,b\in U$ 
and $c\in A^1$, that is, $U$ is $2$-singular.
\end{proof}

The next result gives a lower bound on the 
dimension of the degree-$1$ resonance varieties.

\begin{theorem}
\label{thm:dimres}
Let $A$ be a $\PD_3$ algebra over an algebraically closed field $\k$ (of 
characteristic different from $2$), and let $\nu=\nl(\mu_A)$ be the nullity 
of the associated alternating $3$-form.  If $b_1(A)\ge 4$, then 
\[
\dim \RR^1_{\nu -1}(A)\ge \nu \ge 2.
\]
In particular, $\dim \RR^1_1(A)\ge \nu$. 
\end{theorem}

\begin{proof}
Since $\dim_{\k} A^1\ge 4$ and $\k$ is algebraically closed, 
a result of Sikora \cite[Cor.~20]{Si} implies that $\nl(\mu_A)\ge 2$. 

To prove the other inequality, pick a linear  
subspace $U\subset A^1$ of dimension $\nu$ such that 
$\mu_A(a\wedge b\wedge c)=\varepsilon(abc)=0$ for all $a,b\in U$ 
and $c\in A^1$. By Lemma \ref{lem:sing}, 
the subspace $U$ is isotropic.  
Also, by what we just established,  $\dim U\ge 2$.  
Therefore, by Lemma \ref{lem:isotropic}, 
$U\subseteq \RR^1_{\nu -1}(A)$. Hence,  
$\dim U\le \dim \RR^1_{\nu -1}(A)$, and we are done.
\end{proof}

\subsection{Resonance varieties of $\PD_3$ algebras over $\R$}
\label{subsec:res-real}

Motivated by his study of cut numbers of $3$-manifolds, 
Sikora  made in \cite{Si}  the following conjecture: If $\mu\colon \bwedge^3 V\to \k$ 
is a $3$-form with $\dim V\ge 4$ and if $\ch(\k)\ne 2$, then 
the nullity of $\mu$ is at least $2$ (i.e., $V$ contains a singular plane). 
He noted that the conjecture holds if either $n:=\dim V$ is even or 
equal to $5$, or, as mentioned above, if $\k=\overline{\k}$. 
Nevertheless, work of Draisma and Shaw \cite{DrS10,DrS14} 
implies that the conjecture does not hold for $\k=\R$ and $n=7$.
The following result explains the reason, in terms of resonance varieties. 

\begin{theorem}
\label{thm:real isotropic}
Let $A$ be a $\PD_3$ algebra defined over $\R$.
Then $\RR^1_1(A)\ne \{0\}$, except when $\mu_A$ is one of the 
forms $\mathrm{I}$, $\mathrm{III}$, or $\mathrm{X}_b$ from Appendix \ref{sect:tables}.
\end{theorem}

\begin{proof}
Set $n=b_1(A)$.  If $n\le 2$ everything is clear, so let's assume 
that $n>2$.   We may also assume that $\mu_A$ is irreducible, 
for otherwise, by Corollary \ref{cor:resb1}, $\RR^1_1(A)=A^1$,
and there is nothing to prove. 

Suppose now that $\RR^1_1(A)= \{0\}$, i.e., $\RR^1_1(A)$ contains 
no singular plane. Then, by 
Lemmas \ref{lem:isotropic} and \ref{lem:sing}, $A^1$ contains no 
singular plane. Hence, as shown in \cite[Theorem 2]{DrS14}, the 
formula $(x\times y)\cdot z = \mu_A(x,y,z)$ defines a cross-product on 
$A^1=\R^n$.  In turn, this cross-product yields a division algebra 
structure on $\R^{n+1}$, and so, by a celebrated result of Bott--Milnor 
and Kervaire, we must have $n=3$ or $7$.  An inspection of the tables  from 
Appendix \ref{sect:tables} shows that $\mu_A$ must be equivalent to either 
$\mathrm{III}$ (the associated cross-product on $\R^3$ arises from quaternionic 
multiplication in $\R^4$) or $\mathrm{X}_b$ (as noted in \cite{DrS14}, 
the corresponding cross-product on $\R^7$ arises from octonionic 
multiplication in $\R^8$). This completes the proof.
\end{proof}

The above proof highlights the fact (already alluded to in Remark 
\ref{rem:char-res}) that real resonance varieties may carry 
more refined information than their complex counterparts.  
We make this observation more explicit in the following example.

\begin{example}
\label{ex:sing-iso}
Let $A$ and $A'$ be the real $\PD_3$ algebras corresponding to the 
trivectors $\mathrm{X}_a$ and  $\mathrm{X}_b$.
Then $A\otimes_{\R} \C\cong A'\otimes_{\R} \C$, since 
$\mu_A\sim \mu_{A'}$ over $\C$.  
On the other hand, $A\not\cong A'$ over $\R$, 
since $\mu_A\not\sim \mu_{A'}$ over $\R$, but also 
because $\RR^1_1(A)\ne \{0\}$, yet $\RR^1_1(A')= \{0\}$.  

Note that both $\RR^1_1(A\otimes_{\R} \C)$ and 
$\RR^1_1(A'\otimes_{\R} \C)$ are projectively smooth conics, and thus 
are projectively equivalent over $\C$.  Nevertheless, 
$\RR^1_1(A'\otimes_{\R} \C)=\{x\in \C^7\mid \sum x_i^2=0\}$ has only 
one real point ($x=0$), whereas 
$\RR^1_1(A\otimes_{\R} \C)=\{x\in \C^7\mid x_1x_4+x_2x_5+x_3x_6=x_7^2\}$  
contains, for instance, the real (isotropic) subspace $\{x_4=x_5=x_6=x_7=0\}$.
\end{example}

\section{Pfaffians ideals and resonance}
\label{sect:res pfaff}

In this section we express the resonance varieties of a $\PD_3$ 
algebra $A$ in terms of the Pfaffians of the skew-symmetric matrix  
associated to the boundary map $\delta^1_A$, and determine those 
varieties in bottom depth. 

\subsection{The cochain complex $\bL(A)$}
\label{subsec:aomoto pd3}
Once again, let $A$ be a $\PD_3$ algebra over a field $\k$ of characteristic   
not equal to $2$.  Fix a basis $\{e_1,\dots ,e_n\}$ for $A^1$, 
identify the ring $S=\Sym(A_1)$ with $\k[x_1,\dots, x_n]$, and  
consider the cochain complex $\bL(A)=(A\otimes_\k S,\delta_A)$ 
defined by the BGG correspondence,
\begin{equation}
\label{eq:koszul 3mfd}
\xymatrixcolsep{22pt}
\xymatrix{
A^0\otimes_\k S \ar^{\delta^{0}_A}[r] 
& A^1\otimes_\k  S \ar^{\delta^{1}_A}[r] 
&A^2\otimes_\k  S\ar^{\delta^{2}_A}[r] 
&A^3\otimes_\k  S}.
\end{equation}
Recall from \S\ref{subsec:eqresvar} that the differentials in  $\bL(A)$  
are the $S$-linear maps given by 
$\delta^q(u)=\sum_{j=1}^{n} e_j u \otimes x_j$ for $u\in A^q$.  
In the bases for $A^0,\dots, A^3$ chosen in \S\ref{subsec:alt-forms},
we have that
\begin{align}
\label{eq:delta1 3m}
\delta^0_A(1)&=\sum_{j=1}^{n} e_j \otimes x_j \,, \notag \\
\delta^1_A(e_i)&=\sum_{j=1}^{n} e_j e_i \otimes x_j = 
\sum_{j=1}^{n} \sum_{k=1}^{n}\mu_{jik} e_k^{\vee} \otimes x_j  \, ,  \\
\delta^2_A(e_i^{\vee})&=\sum_{j=1}^{n} e_j e_i^{\vee} \otimes x_j =
\omega \otimes x_i \, . \notag
\end{align}

Observe that the first and third maps have matrices 
$\delta^0_A= \big(x_1 \, \cdots\, x_n\big)$ 
and $\delta^2_A=(\delta^0_A)^{\top}$.%
The most interesting to us is the skew-symmetric matrix 
associated to the boundary map $\delta^1_A$. 

\begin{example}
\label{ex:delta1}
Let $\mu_A=(e^1\wedge e^2+e^3\wedge e^4)\wedge e^5$ be the 
trivector $5_1$ from Appendix \ref{sect:tables}.  Then 
\[
\delta^1_A=\begin{pmatrix}
0&{x}_{5}&0&0&{-{x}_{2}}\\
{-{x}_{5}}&0&0&0&{x}_{1}\\
0&0&0&{x}_{5}&{-{x}_{4}}\\
0&0&{-{x}_{5}}&0&{x}_{3}\\
{x}_{2}&{-{x}_{1}}&{x}_{4}&{-{x}_{3}}&0
\end{pmatrix}.
 \]
\end{example}

\begin{remark}
\label{rem:others}
The matrices $\delta^1_A$ also appear in recent 
work of De~Poi, Faenzi, Mezzetti, and Ranestad \cite{DFMR}, 
as well as Cardinali and Giuzzi \cite{CG}, though in both cases 
the geometric origin and the motivation for studying them is 
very much different from ours. 
\end{remark}

\subsection{Pfaffians and resonance}
\label{subsec:res pf}

By \eqref{eq:r1ka}, each resonance variety $\RR^1_k(A)$ 
is the vanishing locus of the codimension $k$ minors 
of the skew-symmetric matrix $\delta^1_A$. 
More generally, let $\theta$ be a skew-symmetric matrix of size 
$n\times n$ with entries in the polynomial ring $S=\k[x_1,\dots, x_n]$.   
Define the resonance varieties of $\theta$ as 
\begin{equation} 
\label{eq:resmat}
\RR_k(\theta)= V ( I_{n-k} (\theta) ),
\end{equation}
for $0\le k\le n-1$, and set $\RR_n(\theta)=\{0\}$. 
Put another way, the resonance varieties of a skew-symmetric matrix  
$\theta$ are the degeneracy loci of such a matrix.  The next result 
expresses these loci in terms of the Pfaffians of $\theta$.

\begin{theorem}
\label{theorem:be}
Let $\Pf_{2r}(\theta)$ be the ideal of $2r \times 2r$ Pfaffians of an 
$n\times n$ skew-symmetric matrix  $\theta$ with entries in $S$.  Then:
\begin{align} 
\label{eq:res theta}
&\RR_{2k}(\theta) =\RR_{2k+1}(\theta) =V(\Pf_{n-2k}(\theta)),&& \text{if $n$ is even},\\ \notag
&\RR_{2k-1}(\theta) =\RR_{2k}(\theta) =V(\Pf_{n-2k+1}(\theta)),&& \text{if $n$ is odd}.
\end{align}
\end{theorem}

\begin{proof}
As shown by Buchsbaum and Eisenbud  \cite[Cor.~2.6]{BE}, 
the following inclusions hold, for each $r\ge 1$:
\begin{equation} 
\label{eq:pfaff res}
I_{2r}(\theta) \subseteq \Pf_{2r}(\theta) \subseteq \sqrt {I_{2r}(\theta)}, \text{ and } \:
I_{2r-1}(\theta) \subseteq \Pf_{2r}(\theta).
\end{equation}
Consequently, $V(I_{2r-1}(\theta)) = V(I_{2r}(\theta)) = V(\Pf_{2r}(\theta))$, 
and the claim follows. 
\end{proof}

Note that the ideal 
$\Pf_n(\theta)$ is principal, generated by $\pf(\theta)$, the 
maximal Pfaffian of $\theta$, which equals $0$ if $n$ is odd.
Thus, if $n$ is even and $\theta$ is non-singular, then 
$\RR_1(\theta)=\RR_0(\theta)=V(\pf(\theta))$ is a hypersurface, 
while if $\theta$ is singular, then $\RR_1(\theta)=\k^n$.  
On the other hand, if $n$ is odd, then  
$\RR_1(\theta)=\RR_2(\theta)=V(\Pf_{n-1}(\theta))$. 

\begin{remark}
\label{rem:scheme-bis}
We shall view the scheme structure for $\mathbfcal{R}_{k}(\theta)$ 
as being defined by the Pfaffian ideals from \eqref{eq:res theta}.
\end{remark}

Let us return now to the case when $A$ is a $\PD_3$ algebra 
and $\theta=\delta^1_A$ is the boundary map from \eqref{eq:koszul 3mfd}. 
In that case, the matrix $\delta^1_A$ is singular, since $\delta^1_A\circ\delta^0_A=0$. 
Therefore, we have the following chain of inclusions for the varieties 
$\RR^1_k=\RR^1_k(A)$:
\begin{align}
\label{eq:resvarsinc}
&A^1=\RR^1_0=\RR^1_1 \supseteq \RR^1_2=\RR^1_3\supseteq \RR^1_4 =
 \cdots &&\text{if $b_1(A)$ is even},
\\    \notag
&A^1=\RR^1_0\supseteq \RR^1_1= \RR^1_2\supseteq \RR^1_3=\RR^1_4
 \supseteq \cdots &&\text{if $b_1(A)$ is odd}.
\end{align}

\subsection{Bottom-depth resonance}
\label{subsec:bottom}

We conclude this section with a vanishing result for the bottom 
resonance varieties of a $\PD_3$ algebra whose associated $3$-form 
is irreducible.  

\begin{theorem}
\label{thm:rvanish}
Let $A$ be a $\PD_3$ algebra. If  $\mu_A$ has maximal 
rank $n\ge 3$, then 
\begin{equation}
\label{eq:vanish}
\RR^1_{n-2}(A)=\RR^1_{n-1}(A)=\RR^1_{n}(A)=\{0\}.
\end{equation}
\end{theorem}

\begin{proof}
Clearly, $\RR^1_n(A)=\{0\}$. Let $\delta^1=\delta^1_A$ be the 
differential from \eqref{eq:delta1 3m}. 
By \eqref{eq:res theta} and \eqref{eq:r1ka}, we have that 
\[
\RR^1_{n-2}(A) = \RR^1_{n-1}(A) = V(I_1(\delta^1)).
\]

To complete the proof, it suffices to show that $\sqrt{I_1(\delta^1)}=\m$, 
where $\m=\langle x_1,\dots ,x_n\rangle$ is the maximal ideal at $0$.  
By \eqref{eq:koszul 3mfd} all entries of the matrix $\delta^1$ belong to 
$\m$, and so $\sqrt{I_1(\delta^1)}\subseteq \m$.  
Since, by assumption, the form $\mu_A$ has rank $n$, each variable $x_i$ 
occurs in some entry of $\delta^1_A$, and thus equality holds.
\end{proof}

Combining now Theorems  \ref{thm:nonmax} and \ref{thm:rvanish},
we obtain the following immediate corollary.  

\begin{corollary}
\label{cor:bnonmax}
Let $A$ be a $\PD_3$ algebra, and decompose it as $A=B\, \#\, C$, where 
$\mu_B$ is irreducible and $\mu_C=0$.  If $n=\dim A^1$ is at least $3$, then 
$\RR^1_{n-2}(A)=\RR^1_{n-1}(A)=C^1$.
\end{corollary}

\section{Top-depth resonance of \texorpdfstring{$\PD_3$}{PD3} algebras}
\label{sect:top-depth}

In this section we study the geometry of the top-depth resonance varieties 
of a $\PD_3$ algebra, with special emphasis on the case when the associated 
$3$-form satisfies certain genericity conditions.

\subsection{Determinants and Pfaffians}
\label{subsec:det-pfaff}
Let $A$ be a $\PD_3$ algebra over $\k$.   As before, identify $S=\Sym(A_1)$ 
with $\k[x_1,\dots,x_n]$, where $n=b_1(A)$, and let  
$\delta^1=\delta^1_A\colon A^1\otimes_{\k} S\to A^2\otimes_{\k} S$ be the 
first differential in the cochain complex $\bL(A)$.   In the previously chosen 
bases for $A^1$ and $A^2$, the matrix 
of $\delta^1$ is skew-symmetric.  Furthermore, $\delta^1$ is 
singular, since the vector $(x_1,\dots, x_n)$ is in its kernel.  
Hence, both its determinant $\det(\delta^1)$ and its Pfaffian 
$\pf(\delta^1)$ vanish. 

In \cite[Ch.~III, Lemmas 1.2 and 1.3.1]{Tu}, Turaev shows how 
to remedy this situation, so as to obtain well-defined determinant 
and Pfaffian polynomials for the form $\mu=\mu_A$ by looking at 
codimension $1$ minors of the associated matrix $\delta^1$. 

\begin{lemma}[\cite{Tu}]
\label{lem:turaev}
Suppose $n\ge 3$.  There is then a polynomial 
$\Det(\mu)\in S$ such that, if $\delta^1(i;j)$ is the sub-matrix 
obtained from $\delta^1$ by deleting the $i$-th row and $j$-th 
column, then 
\begin{equation*}
\label{eq:detmu}
\det \delta^1(i;j) = (-1)^{i+j}x_ix_j \Det(\mu).
\end{equation*}
Moreover, if $n$ is even, then $\Det(\mu)=0$, while if $n$ is odd, 
then $\Det(\mu)=\Pf(\mu)^2$, where $\pf (\delta^1(i;i)) = (-1)^{i+1} x_i \Pf(\mu)$. 
\end{lemma}

\begin{remark}
\label{rem:deg}
If $n$ is odd, then $\Det(\mu)$ is a homogeneous 
polynomial of degree $n-3$, while $\Pf(\mu)$ is a homogeneous 
polynomial of degree $(n-3)/2$.  
\end{remark}

Let us note the following immediate corollary to Lemma \ref{lem:turaev}.

\begin{corollary}
\label{cor:minors}
With notation as above, let $\m$ be the maximal ideal of $S$ at $0$.  Then 
\[
I_{n-1}(\delta^1)=\begin{cases}
0 & \text{if $n$ is even},\\[2pt]
\m^2\cdot  (\Pf(\mu)^2) & \text{if $n$ is odd}.
\end{cases}
\]
\end{corollary}

We illustrate these notions with a simple example.

\begin{example}
\label{ex:surf}
Let $A=H^{\hdot}(\Sigma_g\times S^1,\k)$, where $\Sigma_g$  
is a Riemann surface of genus $g\ge 1$.  The corresponding $3$-form 
on $A^1=\k^{2g+1}$ is $\mu=\sum_{i=1}^g a_i b_i c$, while $\Pf(\mu)=x_{2g+1}^{g-1}$.  
See also Example \ref{ex:delta1} for the case $g=2$.
\hfill $\Diamond$
\end{example}

\subsection{Generic forms}
\label{subsec:BP-generic}

The alternating $3$-forms from Example \ref{ex:surf} fit into the more 
general class of `generic' $3$-forms, a class introduced and studied 
by Berceanu and Papadima in \cite{BP}. 
For our purposes, it will be enough to consider the case when $n=2g+1$,  
for some $g\ge 1$.  

We say that a $3$-form $\mu\colon \bwedge^3 V\to \k$ is 
{\em BP-generic}\/ if there is an element $v\in V$ such that the $2$-form 
$\gamma_v\in V^*\wedge V^*$ defined by 
\begin{equation}
\label{eq:2form}
\gamma_v(a \wedge b)=\mu_A(a\wedge b\wedge v) \quad\text{for $a,b\in V$}
\end{equation} 
has rank $2g$, that is, $\gamma^{g}_v\ne 0$ 
in $\bwedge^{2g} V^*$.  Equivalently, in a suitable basis for $V$, 
we may write
\begin{equation}
\label{eq:genmu}
\mu=\sum_{i=1}^g a_i \wedge b_i \wedge v + 
\sum w_{ijk}\, z_i \wedge z_j \wedge z_k,
\end{equation}
where each $z_i$  belongs to the span of $a_1,b_1,\dots , a_g, b_g$ in $V$, 
and the coefficients $w_{ijk}$ are in $\k$.  

The following lemma, which was first suggested by S.~Papadima, was 
recorded in \cite[Remark 5.2]{DS} (see also \cite[Remark 4.5] {DPS-mz}). 
For completeness, we supply a proof, in this slightly more general context.

\begin{lemma}
\label{lem:generic}
Assume that $n$ is odd and greater than $1$.  Then 
$\RR^1_1(A)\ne A^1$ if and only if $\mu_A$ is BP-generic.  
\end{lemma}

\begin{proof}
Suppose there is a class $c\in A^1$ such that $c\notin \RR^1_1(A)$.  
Then, for any class  $a\in A^1$ which is not a multiple of $c$, we have 
that $a c\ne 0$.  Letting  $b=(ac)^{\vee}\in A^1$, we infer that 
$\mu_A(a\wedge b\wedge c)$ is non-zero.  
It follows that the $2$-form $\gamma_c$  from \eqref{eq:2form} 
defines a symplectic form on a complementary subspace to the 
vector $c\in A^1$, thereby showing that $\mu_A$ is BP-generic. 
Backtracking through this argument proves the reverse implication. 
\end{proof}

\subsection{The top resonance variety of a $\PD_3$ algebra}
\label{subsec:r11}
We are now in a position to describe fairly explicitly  
the first resonance variety of a $3$-dimensional Poincar\'{e} duality 
algebra.   

\begin{theorem}
\label{thm:respd3}
Let $A$ be a $\PD_3$ algebra over a field $\k$. Set $n=\dim A^1$ 
and let $\mu=\mu_A$ be the associated $3$-form. 
Then 
\begin{equation}
\label{eq:r1 pd3}
\RR^1_1(A)=
\begin{cases}
\emptyset & \text{if\/ $n=0$};\\
\{0\} & \text{if\/ $n=1$ or $n=3$ and $\mu$ has rank $3$};\\
V(\Pf(\mu))  & \text{if\/ $n$ is odd, $n>3$, and $\mu$ is BP-generic};\\
A^1 & \text{otherwise}.
\end{cases}
\end{equation}
\end{theorem}
\begin{proof}

If $n\le 2$, then $\mu=0$, and the conclusion is immediate. 
So suppose $n\ge 3$, and let $\delta^1=\delta^1_A$ be the 
skew-symmetric matrix associated to $\mu$, as in \eqref{eq:koszul 3mfd}.  
Recall from \eqref{eq:r1ka} that $\RR^1_1(A)=V(I_{n-1}(\delta^1))$.

If $n$ is even, then, by Corollary \ref{cor:minors}, 
$I_{n-1}(\delta^1)=0$, and so $\RR^1_1(A)=A^1$.

If $n$ is odd, then again by Corollary \ref{cor:minors}, 
$I_{n-1}(\delta^1)=\m^2\cdot  (\Pf(\mu)^2)$. 
On the other hand, by Lemma \ref{lem:generic}, 
$I_{n-1}(\delta^1)$ is non-zero if and only if $\mu$ is BP-generic.   
In this case, either $n=3$ and so $\Pf(\mu)=1$ and $\RR^1_1(A)=\{0\}$, 
or $n>3$ and $\RR^1_1(A)=V(\Pf(\mu))$ is a hypersurface of degree $(n-3)/2$.
This completes the proof.
\end{proof}

As a corollary, we recover a closely related result, proved by Draisma and Shaw 
in \cite[Thm.~3.2]{DrS10} by very different methods.
 
\begin{corollary}[\cite{DrS10}]
\label{ren:draisma-shaw}
Let $V$ be a vector space of odd dimension $n\ge 5$ over a field $\k$ and let
$\mu\in \bigwedge^3 V^*$. Then the union of all $\mu$-singular planes is either 
all of $V$ or a hypersurface defined by a homogeneous polynomial in
$\k[V]$ of degree $(n-3)/2$.
\end{corollary}

\begin{proof}
Let $A$ be the $\PD_3$ algebra corresponding to $\mu$.  
By Lemmas \ref{lem:isotropic}  and \ref{lem:sing}, the union of all 
$\mu$-singular planes in $A^1=V$ coincides with $\RR^1_1(A)$.  
Suppose that $n$ is odd, $n \ge 5$, and assume $\RR^1_1(A)\ne A^1$ 
(by Lemma \ref{lem:generic}, this means that $\mu$ is BP-generic).
It follows from Theorem \ref{thm:respd3}  
that $\RR^1_1(A)=V(\Pf(\mu))$. By Remark \ref{rem:deg}, $\Pf(\mu)$ 
is a  homogeneous polynomial of degree $(n-3)/2$, and we are done.
\end{proof}

\subsection{Another genericity condition}
\label{subsec:dfmr}

For a trivector $\mu\in \bigwedge^3 V^*$, there is another genericity condition 
studied by De~Poi, Faenzi, Mezzetti, and Ranestad in \cite{DFMR}. 
This condition requires that, for any non-zero vector $v\in V$, the 
bilinear form $\gamma_v$ from \eqref{eq:2form} have rank greater than $2$ 
(this is condition (GC3) from Definition 2.9 in {\it loc.~cit.}, a condition 
which implies that $\mu$ is irreducible).

In the presence of the aforementioned genericity condition, 
a more precise geometric description of the two top resonance 
schemes of the corresponding $\PD_3$ algebra is given in 
\cite[Prop.~4.4]{DFMR}.   We summarize this result in our 
terminology, as follows. 

\begin{theorem}[\cite{DFMR}]
\label{thm:respd3-bis}
Let $A$ be a $\PD_3$ algebra over $\C$, and 
suppose $\mu_A$ is generic in the above sense.  
Writing $n=\dim A^1$, the following hold.
\begin{enumerate}
\item \label{df1}
If $n$ is odd, then $\mathbfcal{R}^1_1(A)$ is a hypersurface of degree $(n-3)/2$ 
which is smooth if $n\le 7$, and singular in codimension $5$ if $n\ge 9$.
\item \label{df2}
If $n$ is even, then $\mathbfcal{R}^1_2(A)$ has codimension $3$ 
and degree $\tfrac{1}{4}\binom{n-2}{3}+1$; it is smooth if $n\le 10$, 
and singular in codimension $7$ if $n\ge 12$.
\end{enumerate}
\end{theorem}
\newpage
\appendix
\section{Resonance varieties of $3$-forms of low rank}
\label{sect:tables}

The following tables list the irreducible $3$-forms $\mu=\mu_A$ of rank $n\le 8$, 
and the corresponding resonance varieties, $\RR_k=\RR^1_k(A)$.  The 
ground field $\k$ is either $\C$ or $\R$, as indicated. For 
simplicity, we will denote a trivector $e^i\wedge e^j \wedge e^k$ as $ijk$. 
We use the classification of $3$-forms of rank at most $8$ of Gurevich \cite{Gu}, 
with further elaborations from \cite{CH, Dj1, Dj2}.   For $n=6$ and $7$, we 
record the way complex orbits split into real orbits, based on the tables of 
Djokovi\'{c} \cite{Dj1}. The computation of the resonance 
varieties was done using the package Macaulay2 \cite{GS}. 

\renewcommand{\arraystretch}{1.4}
\vspace*{1.4pc}
\noindent%
{\small
\begin{tabular}
{|c|c|c|c|c|}
\hline
$\C$
&$\mu$
&$\RR_1$
&$\RR_2$
&$\RR_3$
\\ 
\hline
I 
&$0$ 
&$\emptyset$
&$\emptyset$
&$\emptyset$
\\
\hline
II  
&$123$   
&$0$
&$0$
&$0$
\\
\hline
III
&$125+345$ 
&$\{x_5=0\}$
&$\{x_5=0\}$
&0 
\\
\hline
\end{tabular}

\vspace*{1.4pc}
\noindent
\begin{tabular}
{|c|c|>{\centering\arraybackslash}m{1.05in}|c|>{\centering\arraybackslash}m{3.25in}|c|}
\hline
$\C$ & $\R$
&$\mu$
&$\RR_1$
&$\RR_2= \RR_3$
&$\RR_4$
\\
\hline
IV &
&$135+234+126$ 
&$\k^6$
&$\{x_1=x_2=x_3=0\}$
&0
\\
\hline
V & a
&$123+456$   
&$\k^6$   
&$\{x_1=x_2=x_3=0\}\cup  \{x_4=x_5=x_6=0\}$
& 0
\\
\cline{2-6}
 & b 
&$-135+146+236+245$
&$\k^6$   
&$V(x_1^2+x_2^2, x_3^2+x_4^2, x_5^2+x_6^2,x_4x_5-x_3x_6,
x_3x_5+x_4x_6, \linebreak \hspace*{12pt} x_2x_5-x_1x_6,
x_1x_5+x_2x_6,x_2x_3-x_1x_4,x_1x_3+x_2x_4)$
& 0
\\
\hline
\end{tabular}

\vspace*{1.4pc}
\noindent
\resizebox{6.2in}{!} {%
\begin{tabular}
{|c|c|>{\centering\arraybackslash}m{1.5in}|c|>{\centering\arraybackslash}m{1.7in}|c|}
\hline
$\C$ & $\R$
&$\mu$
&$\RR_1=\RR_2$
&$\RR_3=\RR_4$
&$\RR_5$ 
\\ 
\hline
VI   
&
&$123+145+167$  
&$\{x_1=0\}$
&$\{x_1=0\}$
& 0
\\
\hline
VII 
&
&$125+136+147+234$
&$\{x_1=0\}$
&$\{x_1=x_2=x_3=x_4=0\}$
& 0
\\
\hline
VIII 
&a
&$134+256+127$
&$\{x_1=0\}\cup \{x_2=0\}$
&$\{x_1=x_2=x_3=x_4=0\}\cup\linebreak 
\{x_1=x_2=x_5=x_6=0\} $
& 0
\\
\cline{2-6}
&b
&  $-135+146+236+245+127$
&$\{x_1^2+x_2^2=0\}$
& $V(x_1,x_2,x_3^2+x_4^2,x_5^2+x_6^2,x_3x_5+x_4x_6,x_4x_5-x_3x_6)$
& 0
\\
\hline
IX  
&a
&$125+346+137+247$
&$\{x_1x_4+x_2x_5=0\}$
&$V(x_7^2-x_3x_6,x_1,x_2,x_4,x_5)$
& 0
\\
\cline{2-6}
&b
& $-135+146+236+245+127+347$
& $\{x_1x_3+x_2x_4=0\}$
& $V(x_7^2-x_5x_6,x_1,x_2,x_3,x_4)$
& 0
\\
\hline
X  
&a
&$123+456+147+257+367$
&$\{x_1x_4+x_2x_5+x_3x_6=x_7^2\}$
&0
& 0
\\
\cline{2-6}
&b
&$-135+146+236+245+127+347+567$
&$\{x_1^2+x_2^2+x_3^2+x_4^2+x_5^2+x_6^2+x_7^2=0\}$
& 0 
& 0
\\
\hline
\end{tabular}
}
}
\newpage
\noindent
\resizebox{6.2in}{!} {%
\begin{tabular}
{|c|>{\centering\arraybackslash}m{1.4in}|c|>{\centering\arraybackslash}m{1.9in}|
>{\centering\arraybackslash}m{1.95in}|c|}
\hline
$\C$
&$\mu$
&$\RR_1$
&$\RR_2=\RR_3$
&$\RR_4=\RR_5$
&$\RR_6$
\\ 
\hline
XI
& $147+257+367+358$  
&$\C^8$
&$\{x_7=0\}$
&$\{x_3=x_5=x_7=x_8=0\}\cup\linebreak 
\{x_1=x_3=x_4=x_5=x_7=0\}$
&0
\\
\hline
XII
& $456+147+257+367+358$  
&$\C^8$
&$\{x_5=x_7=0\}$
&$\{x_3=x_4=x_5=x_7=x_1x_8+x_6^2=0\}$
&0
\\
\hline
XIII
& $123+456+147+358$ 
&$\C^8$
&$\{x_1=x_5=0\}\cup\linebreak \{x_3=x_4=0\}$
&$\{x_1=x_3=x_4=x_5=x_2x_6+x_7x_8=0\}$
&0
\\
\hline
XIV
& $123+456+147+257+358$ 
&$\C^8$
&$\{x_1=x_5=0\}\cup\linebreak 
\{x_3=x_4=x_5=0\}$  
&$\{x_1=x_2=x_3=x_4=x_5=x_7=0\}$
&0
\\
\hline
XV
&  $123+456+147+257+367+358$ 
&$\C^8$
&$\{x_3=x_5=x_1x_4-x_7^2=0\}$ 
&$\{x_1=x_2=x_3=x_4=x_5=x_6=x_7=0\}$
&0
\\
\hline
XVI
&  $147+268+358$    
&$\C^8$
&$\{x_1=x_4=x_7=0\}\cup \{x_8=0\}$  
&$\{x_1=x_4=x_7=x_8=0\}\cup\linebreak  
\{x_2=x_3=x_5=x_6=x_8=0\}$
&0
\\
\hline
XVII
&  $147+257+268+358$ 
&$\C^8$
&$\{x_7=x_8=0\}\cup \{x_2=x_5=x_8=0\}\cup
\{x_1=x_4=x_7=0\}$ 
&$\{x_1=x_2=x_4=x_5=x_7=x_8=0\}\cup  
\{x_2=x_3=x_5=x_6=x_7=x_8=0\}$
&0
\\
\hline
XVIII
&   $456+147+257+268+358$ 
&$\C^8$
&$\{x_5=x_8=x_4x_6-x_2x_7=0\}\cup\linebreak  
\{ x_4=x_7=x_5^2-x_1x_8=0\}$ 
&$\{x_1=x_2=x_4=x_5=x_6=x_7=x_8=0\} \cup 
\{x_2=x_3=x_4=x_5=x_6=x_7=x_8=0\}$
&0
\\
\hline
XIX
&  $147+257+367+268+358$ 
&$\C^8$
&$\{x_2-x_3=x_5-x_6=x_7-x_8=0\} \cup
\{x_2+x_3=x_5+x_6=x_7+x_8=0\} \cup 
\{x_7=x_8=0\} \cup \{x_1=x_4=x_7=0\}$  
&$\{x_1=x_4=x_7=x_8=x_2-x_3=x_5-x_6=0\} 
\cup \{x_1=x_4=x_7=x_8=x_2+x_3=x_5+x_6=0\} 
\cup \{x_2=x_3=x_5=x_6=x_7=x_8=0\}$
&0
\\
\hline
XX
&  $456+147+257+367+268+358$ 
&$\C^8$
&$\{x_5-x_6=x_7-x_8=x_4x_6-x_2x_7+ x_3x_7=0\} \cup 
\{x_5+x_6=x_7+x_8=x_4x_6-x_2x_7-x_3x_7=0\} \cup \linebreak
\{x_7=x_4=x_5^2-x_6^2-x_1x_8=0\}$ 
&$\{x_1=x_4=x_5=x_6=x_7=x_8=x_2-x_3=0\} \cup 
\{ x_1=x_4=x_5=x_6=x_7=x_8=x_2+x_3=0\} \cup 
\{ x_2=x_3=x_4=x_5=x_6=x_7=x_8=0\}$
&0
\\
\hline
XXI
&  $123+456+147+268+358$ 
&$\C^8$
&$\{x_3x_5+x_2x_6+x_7x_8=x_1=x_4=0\} 
\cup
\{x_1x_4+x_8^2=x_1x_3-x_6x_8=x_1x_2+x_5x_8=
x_4x_6+x_3x_8=x_4x_5-x_2x_8=x_3x_5+x_2x_6=0\}$  
&$\{x_1=x_2=x_3=x_4=x_5=x_6=x_8=0\}$   
&0
\\
\hline
XXII
&  $123+456+147+257+268+358$ 
&$\C^8$
&$\{f_1=\cdots = f_{20}=0\}$  
&$0$
&0
\\
\hline
XXIII
&  $123+456+147+257+367+268+358$ 
&$\C^8$
&$\{g_1=\cdots = g_{20}=0\}$ 
&$0$
&0
\\
\hline
\end{tabular}
}
\renewcommand{\arraystretch}{1.0}
\vspace*{4pt}

Note:  In XXII and XXIII, the polynomials $f_i$ and $g_i$ 
are homogeneous of degree $3$.  The varieties cut out by each of these 
two sets of polynomials have codimension $3$. 

\section*{Acknowledgement}
\label{sec:ack}
I would like to thank the referee for very helpful remarks and suggestions 
that led to improvements in both the substance and the exposition of the paper.

\newcommand{\arxiv}[1]
{\texttt{\href{http://arxiv.org/abs/#1}{arXiv:#1}}}
\newcommand{\arxi}[1]
{\texttt{\href{http://arxiv.org/abs/#1}{arxiv:}}
\texttt{\href{http://arxiv.org/abs/#1}{#1}}}
\newcommand{\arxx}[2]
{\texttt{\href{http://arxiv.org/abs/#1.#2}{arxiv:#1.}}
\texttt{\href{http://arxiv.org/abs/#1.#2}{#2}}}
\newcommand{\doi}[1]
{\texttt{\href{http://dx.doi.org/#1}{doi:\nolinkurl{#1}}}}
\renewcommand{\MR}[1]
{\href{http://www.ams.org/mathscinet-getitem?mr=#1}{MR#1}}
\newcommand{\MRh}[2]
{\href{http://www.ams.org/mathscinet-getitem?mr=#1}{MR#1 (#2)}}
\newcommand{\Zbl}[1]
{\href{http://zbmath.org/?q=an:#1}{Zbl #1}}

\end{document}